\documentclass[final]{siamltex}
\usepackage{amsmath}
\usepackage{amssymb}
\usepackage{float}
\usepackage{comment}
\usepackage{graphicx}
\usepackage{hyperref} %Darstellung von Web-Adressen
\hypersetup{draft=false, colorlinks=true, linkcolor=black, urlcolor=blue, citecolor=black}

\newtheorem{remark}[theorem]{Remark}

\allowdisplaybreaks

\def \RR {{\mathbb R}}

\def \< {\langle}
\def \> {\rangle}

\def \RN {\mathbb{R}_+^N}
\def \Drm {\mathrm{D}}

\title{On the linearization of identifying the stored energy function of a hyperelastic material from full knowledge of the displacement field}

\author{
Julia Seydel\thanks{Department of Mathematics, Saarland University, PO Box 15 11 50, 66041 Saarbr\"ucken, Germany ({\tt julia.seydel@math.uni-sb.de}).} \and
Thomas Schuster\thanks{Department of Mathematics, Saarland University, PO Box 15 11 50, 66041 Saarbr\"ucken, Germany ({\tt thomas.schuster@num.uni-sb.de}),
correspondent author.}}

\begin{document}

\maketitle

\begin{abstract}
We consider the nonlinear, \emph{inverse problem} of computing the stored energy function of a hyperelastic material from the full knowledge of the displacement field.
The displacement field is described as solution of the nonlinear, dynamic, elastic wave equation, where the first Piola-Kirchhoff stress tensor is given as the gradient
of the stored energy function. We assume that we have a dictionary at hand such that the energy function is given as a conic combination of the dictionary's elements.
In that sense the mathematical model of the \emph{direct problem} is the nonlinear operator that maps the vector of expansion coefficients to the solution of the
hyperelastic wave equation. In this article we summarize some continuity results for this operator and deduce its Fr\'{e}chet derivative as well as the adjoint of this
derivative. Since the stored energy function encodes mechanical properties of the underlying, hyperelastic material, the considered inverse problem is of highest interest for 
structural health monitoring systems where defects are detected from boundary measurements of the displacement field. For solving the inverse problem iteratively by
the Landweber method or Newton type methods, the knowledge of the Fr\'{e}chet derivative and its adjoint is of utmost importance. 
\end{abstract}


\pagestyle{myheadings}
\thispagestyle{plain}
\markboth{J. Seydel and T. Schuster}{On the linearization of identifying the stored energy function for hyperelastic materials}

%%%%%%%%%%%%%%%% EINLEITUNG %%%%%%%%%%%%%%%%%%%%%%%%%%%%%

\section{Introduction}

Starting point of our inverse problem is Cauchy's equation of motion for an elastic material of continuum mechanics (elastic wave equation)
\begin{equation}\label{cauchy}
\rho (x) \ddot{u}(t,x) - \nabla\cdot P(t,x) = f(t,x),\,
\end{equation}
where $t\in[0,T]$ denotes time and $x\in \Omega\subset\RR^3$ a point in a bounded, open domain $\Omega$ in $\RR^3$. Furthermore, $\rho(x)$ denotes the mass density in $x\in\Omega$, 
$f:[0,T]\times \Omega\to \RR^3$ is an external body force and $P: [0,T]\times \Omega\to \RR^{3\times 3}$ is the first Piola-Kirchhoff stress tensor. 
The vector field $u: [0,T]\times \Omega\to \RR^3$ is then the displacement vector, that the structure perceives under stress $P$ and external force $f$ in position $x\in\Omega$ at time $t\in[0,T]$.
We assume, that 
\begin{equation}\label{massendichte}
 0 < \inf_{x\in\Omega}\rho(x)=:\rho_{\min}\leq\rho(x)\leq\rho_{\max}:=\sup_{x\in\Omega} < \infty.
\end{equation}
\emph{Hyperelastic materials} are materials with $P(t,x)=\nabla_Y C(x,Ju(t,x))$ for a function $C=C(x,Y)$, $Y\in \RR^{3\times 3}$ with $\mathrm{det}\, Y>0$, which is the \emph{stored energy function}. Here, the derivative $\nabla_{Y}$ is to be understood componentwise and $Ju(t,x)=(\partial_{j}u_{i}(t,x))_{i,j=1,2,3}$ represents the displacement gradient in $t\in[0,T]$ and $x\in\Omega$.\\
Equipped with appropriate initial and boundary value conditions we get the system
\begin{equation}\label{cauchy-hyper}
\rho (x) \ddot{u}(t,x) - \nabla\cdot \nabla_Y C(x,Ju(t,x)) = f(t,x)\,
\end{equation}
for $t\in[0,T]$ and $x\in\Omega\subset\RR^3$ with given initial values
\begin{eqnarray}
\label{anfangswert1}
u(0,\cdot) &=& u_0\in H^{2}(\Omega,\RR^3),\\[1ex]
\label{anfangswert2}
\dot{u}(0,\cdot) &=& u_1\in H^{1}(\Omega,\RR^3)
\end{eqnarray}
and homogeneous Dirichlet boundary values 
\begin{equation}\label{randwert}
 u(t,x) = 0,\;\;t\in[0,T],\;\;x\in\partial\Omega.
\end{equation}
Equation (\ref{cauchy-hyper}) describes the behavior of hyperelastic materials such as e.g. carbon fibre reinforced composits (CFRC). That is why it plays a prominent role in materials science.
Inverse problems associated with equations (\ref{cauchy}) or (\ref{cauchy-hyper}) furthermore have applications in Structural Health Monitoring (SHM), see \cite{Giu08}. 
SHM systems consist of a number of actors that generate guided waves propagating through the structure and are measured by sensors which are applied to the structure's surface. The aim is the early detection of defects from the sensor measurements. Mathematically this leads to the inverse problem of computing material properties from boundary data of the displacement field. This is a nonlinear parameter identification problem for (\ref{cauchy}) or (\ref{cauchy-hyper}). In this article we specifically consider the inverse problem of reconstructing the stored energy function $C(x,Y)$ from the full displacement field $u (t,x)$. Since any defect of the structure affects $C$, the identification of a spatially variable $C$ might be an appropriate feature for damage detection of hyperelastic materials. We refer to standard textbooks like
\cite{CIARLET:88, Holzapfel200003, MARSDEN;HUGHES:83} for detailed introductions and derivations of the given equations.\\
There are numerous publications on inverse identification problems in elastic media for different settings. In \cite{HAEHNER:02} H\"ahner has analysed the problem of reconstructing the 
mass density in inhomogeneous, isotropic media from far field data. 
The linear sampling method, introduced by Colton and Kirsch in \cite{CK:1996} for detection of reverberant scatterers, was applied and implemented for the isotropic 
Navier Lam\'{e} equation by Bourgeois and others in \cite{BLL:2011}. The method describes a possibility to detect defects in isotropic materials and damages, which are represented by such a scatterer. Based on this the method was used in \cite{BL13} 
for the identification of cracks. Inverse problems are also object of \cite{BEYLKIN;BURRIDGE:90}. 
Sedipkov \cite{SEDIPKOV:11} considers the inverse problem to compute the acoustic impedance in an inhomogeneous, elastic medium from Cauchy data. An extensive investigation of inverse problems for acoustic and elastic waves is 
\cite{SPS84}, where problems of determining mechanical properties of inhomogeneous media as well as problems of identification the nature of a radiating wave source from boundary data are considered. The reconstruction of Lam\'{e} coefficients from Cauchy data for an isotropic material in 
2D and 3D is investigated in \cite{IUY:2012, IY:2011}. It was possible to obtain uniqueness results when Dirichlet data are available on a part of the boundary. The Lipschitz continuity of the Dirichlet-to-Neumann mapping in the case of isotropic, inhomogeneous materials could be demonstrated in the article \cite{BERETTA;ET;AL:2014}. 
Uniqueness results for anisotropic material tensors are contained in \cite{Kaltenbacher200711} for the case, that the material tensor can be represented as spatially constant conic combination of given tensors. 
The results of this paper have been extended essentially for the identification of material parameters from sensor data in \cite{SCHUSTER;WOESTEHOFF:14}. Important contributions related to uniqueness results for the identification of the Lam\'{e} constants in isotropic, inhomogeneous media for arbitrary dimensions from complete Cauchy data are the articles by Nakamura and Uhlmann \cite{NU94,NU95}. 
Uniqueness results for inverse problems for elastic, anisotropic media are also included in \cite{MR:2006}. An algorithm, that ensures both the conservation of the total energy and the conservation of momentum and angular momentum, is presented in \cite{ST92}. 
In \cite{BG04} uniqueness results are given for the determination of the shear modulus from a finite number of linearly independent displacement fields in two dimensions. 
The reconstruction of an anisotropic elasticity tensor from 
a finite number of displacement fields for the linear, stationary elasticity equation is the topic of \cite{BMU15}. A comprehensive overview of various inverse problems in the field of elasticity offers the article \cite{BC:2005}. 
In contrast to the publications mentioned before we consider the identification of stored energy functions, which are spatially variable, from time-dependent boundary data. The article \cite{KIRSCH;RIEDER:14} by Kirsch and Rieder can be seen in some sense as an analogon for the acoustic wave equation.\\
We specify the inverse problem  to be investigated in this article. Inspired by Kal\-ten\-ba\-cher and Lorenzi \cite{Kaltenbacher200711} and following the authors of \cite{SCHUSTER;WOESTEHOFF:14,WOESTEHOFF;SCHUSTER:15} we suppose to have a dictionary $\{C_1,C_2,\ldots,C_N\}$ consisting of functions\footnote{The specific choice of the functions $C_K$ is not subject of this article.} $C_K=C_K (x,Y)$, $K=1,\ldots,N$, given such that
\begin{equation}\label{con-comb}
 C (x,Y)  = \sum_{K=1}^{N}\alpha_{K} C_{K} (x,Y)
\end{equation}
with positive constants $\alpha_K> 0$, $K=1,\ldots, N$. In that way the searched function $C$ is projected onto a dictionary consisting of physically meaningful elements such as polyconvex functions, see \cite{Ball1977, Holzapfel200003}. The hyperelastic wave equation then is given as
\begin{equation}\label{cauchy-hyper-comb} 
 \rho(x)\ddot{u}(t,x) - \sum_{K=1}^{N}\alpha_{K}\nabla\cdot [\nabla_{Y}C_{K}(x,Ju(t,x))] = f(t,x)
\end{equation}
for $t\in[0,T]$ and $x\in\Omega\subset\RR^3$. We consider the following inverse problem:\\[1ex]
\textbf{(IP)}\quad Given $(f,u_0,u_1)$ as well as the displacement field $u(t,x)$ for $t\in [0,T]$ and $x\in \Omega$, compute the coefficients $\alpha=(\alpha_1,\ldots,\alpha_N)\in\RN$,
such that $u$ satisfies the initial boundary value problem (IBVP) (\ref{cauchy-hyper-comb}), (\ref{anfangswert1})--(\ref{randwert}).\\[1ex]
Denoting by $\mathcal{T} : \Drm (\mathcal{T})\subset \RN \to L^2 (0,T;H^1 (\Omega,\RR^3))$ the \emph{forward operator}, which maps a vector $\alpha\in\Drm (\mathcal{T}) $ to the unique solution of the IBVP (\ref{cauchy-hyper-comb}), (\ref{anfangswert1})--(\ref{randwert}) for $(f,u_0,u_1)$ fixed, then (IP) is just given as the nonlinear operator equation
\[   \mathcal{T} (\alpha) = u\,.   \]
Here, $\Drm (\mathcal{T})$ denotes the domain of $\mathcal{T}$ to be specified in Section 3. The article shows that $\mathcal{T}$ is Fr\'{e}chet differentiable for $\alpha\in \mathrm{int} \big(\Drm (\mathcal{T})\big)$ and gives representations for $\mathcal{T}' (\alpha)$ as well as for the adjoint operator $\mathcal{T}' (\alpha)^*$.\\[1ex]
\emph{Outline}. Section 2 provides all mathematical ingredients and tools which are necessary to prove the main results of the article. Particularly we summarize an existing uniqueness result
for the solution of the IBVP (\ref{cauchy-hyper-comb}), (\ref{anfangswert1})--(\ref{randwert}) (Theorem \ref{theorem21imp}). Section 3 represents the core of the article containing the main results.
First we deduce the G\^{a}teaux derivative of $\mathcal{T}$ (Lemma \ref{frechet-system}), show its continuity (Theorem \ref{gateaux-stetig}) and finally prove the Fr\'{e}chet differentiability 
(Theorem \ref{thm-glm-konv}). In Section 4 we furthermore derive a representation of the adjoint operator $\mathcal{T}' (\alpha)^*$ (Theorem \ref{Tadjungiert}). Section 5 concludes the article.\\

%%%%%%%%%%%%%%%%%%%%%%%%%%%%%%%%%%%%%%%%%%%%Section 2%%%%%%%%%%%%>%%%%%%%%%%%%%%%%%%%%%%%%%%%%%%%%%%%%%%%%

\section{Setting the stage}\label{secVor}

%%%%%%%%%%%%%%%%%%%%%%%%%%%%%%%%%%%%%%%%%%%%%%%%%%%%%%%%%%%%%%%%%%%%%%%%%%%%%%%%%%%%%%%%%%%%%%%%%%%%%%%%%

We start by recapitulating an existing uniqueness result for the IBVP (\ref{cauchy-hyper-comb}), (\ref{anfangswert1})--(\ref{randwert}) from \cite{WOESTEHOFF;SCHUSTER:15} as well as some
important estimates that we need to prove our main results.\\
We assume that $C_{K}:\;\Omega\times\RR^{3\times3}\to\;\RR$ satisfies the conditions $C_{K}(x,0)=0$ and $\nabla_{Y}C_{K}(x,0)=0$ for all $K=1,...,N$ and $x\in\Omega\subset\RR^3$. 
We restrict the nonlinearity of the functions $C_{K}$ and hence of $C$ by supposing, that there are positive constants $\kappa_{K}^{[0]},\;\kappa_{K}^{[1]},\;\mu_{K}^{[0]},\;\mu_{K}^{[1]}$ for $K=1,...,N$, such that 
\begin{equation}\label{bed1}
 \kappa_{K}^{[0]}\|Y\|_{F}^{2} \leq C_{K}(x,Y) \leq \mu_{K}^{[0]}\|Y\|_{F}^{2}
\end{equation}
and 
\begin{equation}\label{bed2}
 \kappa_{K}^{[1]}\|H\|_{F}^{2} \leq \langle\langle H|\nabla_{Y}\nabla_{Y}C_{K}(x,Y)H\rangle\rangle \leq \mu_{K}^{[1]}\|H\|_{F}^{2}
\end{equation}
hold for all $H,Y\in\RR^{3\times3}$ and for $x\in\Omega$ almost everywhere. By $\|\cdot\|_{F}$ we denote the Frobenius norm induced by the inner product of matrices 
$$\langle\langle A|B \rangle\rangle:=\mbox{tr}(A^{\top}B)\qquad \mbox{for }A,B\in\RR^{3\times3}.$$
Additionally we require the existence and boundedness of higher derivatives of $C_K$ with respect to $Y$. More specifically we assume, that there are constants $\mu_{K}^{[2]},...,\mu_{K}^{[7]}$ for $K=1,...,N$ with 
\begin{equation}\label{beschr1}
 \|\partial_{Y_{pq}}\partial_{Y_{ij}}\partial_{Y_{kl}}C_{K}\|_{L^{\infty}(\Omega\times\RR^{3\times3})}\leq\mu_{K}^{[2]}
\end{equation}
\begin{equation}\label{beschr2}
 \|\partial_{Y_{ab}}\partial_{Y_{pq}}\partial_{Y_{ij}}\partial_{Y_{kl}}C_{K}\|_{L^{\infty}(\Omega\times\RR^{3\times3})}\leq\mu_{K}^{[3]}
\end{equation}
\begin{equation}\label{beschr3}
 \|\partial_{l}\partial_{Y_{kl}}C_{K}\|_{L^{\infty}(\Omega\times\RR^{3\times3})}\leq\mu_{K}^{[4]}
\end{equation}
\begin{equation}\label{beschr4}
 \|\partial_{Y_{ij}}\partial_{l}\partial_{Y_{kl}}C_{K}\|_{L^{\infty}(\Omega\times\RR^{3\times3})}\leq\mu_{K}^{[5]}
\end{equation}
\begin{equation}\label{beschr5}
 \|\partial_{l}\partial_{Y_{ij}}\partial_{Y_{kl}}C_{K}\|_{L^{\infty}(\Omega\times\RR^{3\times3})}\leq\mu_{K}^{[6]}
\end{equation}
\begin{equation}\label{beschr6}
 \|\partial_{Y_{pq}}\partial_{l}\partial_{Y_{ij}}\partial_{Y_{kl}}C_{K}\|_{L^{\infty}(\Omega\times\RR^{3\times3})}\leq\mu_{K}^{[7]}
\end{equation}
for $a,b,i,j,k,l,p,q=1,2,3$ and $K=1,...,N$. Furthermore, we require
\begin{equation}\label{vertauschen-partial}
 \partial_{Y_{ij}}\partial_{l}\partial_{Y_{kl}}C(x,Y)=\partial_{l}\partial_{Y_{ij}}\partial_{Y_{kl}}C(x,Y)
\end{equation}
for all $i,j,k,l=1,2,3$, which holds true if e.g. $C\in\mathcal{C}^{4}(\Omega\times\RR^{3\times3})$. The mapping $Y\to C_{K}(x,Y)$ is supposed to be three times continuously differentiable for $x\in\Omega$
almost everywhere.\\ 
Furthermore, we restrict the set of admissible coefficient vectors $\alpha=(\alpha_{1},...,\alpha_{N})^{\top}\in\RR_+^N$ of the conic combination (\ref{con-comb}) by assuming 
\begin{eqnarray*}
 && \alpha\in \mathcal{C}((\kappa^{[a]})_{a=1,2},(\mu^{[b]})_{b=1,...,7})\\[1ex]
 &:=& \begin{Bmatrix}
      \alpha\in (0,\infty )^{N}:\sum_{K=1}^{N}\alpha_{K}\kappa_{K}^{[a]}\geq \kappa^{[a]},\;\sum_{K=1}^{N}\alpha_{K}\mu_{K}^{[b]}\leq \mu^{[b]}\\[1ex]
      \mbox{ for all}\; a=1,2\;\mbox{ and }\;b=1,...,7
     \end{Bmatrix}.
\end{eqnarray*}
This set is coupled to the nonlinearity conditions of $C_{K}$ (\ref{bed1})--(\ref{beschr6}) via the constants $\mu^{[b]}$.\\
Finally, we define the set of admissible solutions $u$ of IBVP (\ref{cauchy-hyper})--(\ref{randwert}). For given constants 
$M_{i}$, $i=0,...,3$ we set
\begin{eqnarray}\label{bedA}
 && \mathcal{A} := \mathcal{A}(M_{0},M_{1},M_{2},M_{3}) \nonumber \\[1ex]
 &&:= \begin{Bmatrix}
      u\in L^{\infty}((0,T)\times\Omega,\RR^3)\cap W^{1,\infty}((0,T),H^{1}(\Omega,\RR^3)):\\[1ex]
      \|\partial_{l}\partial_{j}u\|_{L^{\infty}((0,T),L^{2}(\Omega,\RR^3))} < M_{0},\;\|\partial_{l}\dot{u}_{k}\|_{L^{\infty}((0,T)\times\Omega)} < M_{1},\\[1ex]
      \|\partial_{l}\partial_{j}\dot{u}_{k}\|_{L^{\infty}((0,T)\times\Omega)} < M_{2},\;\|\partial_{l}\partial_{j}u_{k}\|_{L^{\infty}((0,T),\Omega)} < M_{3}\\[1ex]
      \mbox{for all}\; i,j,k,l=1,2,3
     \end{Bmatrix}.
\end{eqnarray}
If we e.g. assume, that $\partial\Omega$, $f$, $u_{0}$, $u_{1}$ and $C_{K}$ are sufficiently smooth, then $u\in\mathcal{A}$ holds true.\\
All these constraints are necessary to prove the following uniqueness result for the solution of the IBVP (\ref{cauchy-hyper-comb}), (\ref{anfangswert1})--(\ref{randwert}) for given
$\alpha\in \mathcal{C}((\kappa^{[a]})_{a=1,2},(\mu^{[b]})_{b=1,...,7})$, which has been presented in \cite{WOESTEHOFF;SCHUSTER:15}.\\
\begin{theorem}[{\cite[Theorem 2.1]{WOESTEHOFF;SCHUSTER:15}}]\label{theorem21imp}
 Let $u$, $\bar{u}$ be two solutions to the initial boundary value problem (\ref{cauchy-hyper-comb}), (\ref{anfangswert1})--(\ref{randwert}) corresponding to the parameters, initial values and right-hand sides $(\alpha, u_{0}, u_{1}, f)$ and $(\bar{\alpha}, \bar{u}_{0}, \bar{u}_{1}, \bar{f})$, respectively.  
 Furthermore, assume that $u,\bar{u}\in\mathcal{A}$. If, in addition, the condition
 \begin{equation}\label{dim}
  \frac{7}{8}\mu < \kappa < \frac{9}{8}\mu
 \end{equation}
 is satisfied for
 \begin{equation}\label{kappamu}
\kappa:=\sum_{K=1}^{N}\alpha_{K}\kappa_{K}^{[1]}\;\;\mbox{and}\;\; \mu:=\sum_{K=1}^{N}\alpha_{K}\mu_{K}^{[1]} 
\end{equation}
and if there are constants $\kappa(\alpha)$ and $\mu(\alpha)$, so that
\begin{equation}\label{kappamuungl}
\kappa\geq\kappa(\alpha)>0\;\; \mbox{and}\;\; \mu\leq\mu(\alpha),
\end{equation}
then there exist constants $\bar{C}_{0}$, 
$\bar{C}_{1}$ and $\bar{C}_{2}$, such that the stability estimate
\begin{eqnarray*}
 &&\bigg[\rho\|(\dot{u}-\dot{\bar{u}})(t,\cdot)\|_{L^{2}(\Omega,\RR^3)}^{2} + \kappa(\alpha)\|(Ju-J\bar{u})(t,\cdot)\|_{L^{2}(\Omega,\RR^{3\times 3})}^{2}\\[1ex]
 && + \rho\|(\ddot{u}-\ddot{\bar{u}})(t,\cdot)\|_{L^{2}(\Omega,\RR^3)}^{2} + \kappa(\alpha)\|(J\dot{u}-J\dot{\bar{u}})(t,\cdot)\|_{L^{2}(\Omega,\RR^{3\times 3})}^{2}\\[1ex]
 &&\;\; +\|(u-\bar{u})(t,\cdot)\|_{H^{2}(\Omega,\RR^3)}^{2}\bigg]^{\frac{1}{2}}\\
 &&\leq \bar{C}_{0}\bigg[\mu(\alpha)\|(u_{0}-\bar{u}_{0})(t,\cdot)\|_{H^{2}(\Omega,\RR^3)}^{2} + \|(u_{1}-\bar{u}_{1})(t,\cdot)\|_{H^{1}(\Omega,\RR^3)}^{2}\bigg]^{\frac{1}{2}}\\[1ex]
 && + \bar{C}_{1} \|(f-\bar{f})\|_{W^{1,1}((0,T),L^{2}(\Omega,\RR^3))} + \bar{C}_{2}\|\alpha-\bar{\alpha}\|_{\infty}
\end{eqnarray*}
is valid for all $t\in(0,T)$.
Thereby, the constants $\bar{C}_{0}$, $\bar{C}_{1}$ and $\bar{C}_{2}$ only depend on $T$, $M_{0}$, $M_{1}$, $M_{2}$, $M_{3}$,
\begin{equation}
 \bar{C}(\alpha) := \sum_{K=1}^{N}\alpha_{K}\mu_{K}^{[2]}\bigg(\sum_{K=1}^{N}\alpha_{K}\kappa_{K}^{[1]}\bigg)^{-1}
\end{equation}
and 
\begin{equation}
 \hat{C}(\alpha) := \frac{\hat{K}}{1-\sqrt{1-\epsilon}}\sum_{K=1}^{N}\alpha_{K}\mu_{K}^{[1]}\bigg(\sum_{K=1}^{N}\alpha_{K}\kappa_{K}^{[1]}\bigg)^{-2},
\end{equation}
where $0<\epsilon<1$ is a constant, whose existence is ensured by inequality (\ref{dim}). The constant $\hat{K}>0$ is defined by the continuity of the embedding $H_{0}^{2,1}(\Omega,\RR^3):=H^{2}(\Omega,\RR^3)\cap H_{0}^{1}(\Omega,\RR^3)\hookrightarrow H^{2}(\Omega,\RR^3)$,
\[\|g\|_{H^{2}(\Omega,\RR^3)}\leq \hat{K}\|g\|_{H_{0}^{2,1}(\Omega,\RR^3)}=\hat{K}\bigg(\sum_{k=1}^{3}\int\limits_{\Omega}\sum_{l=1}^{3}\sum_{j=1}^{3}(\partial_{i}\partial_{j}g_{k}(x))^{2}dx\bigg)^{\frac{1}{2}}\]
for all $g\in H_{0}^{2,1}(\Omega,\RR^3)$. Moreover, the constants $\bar{C}_{0}$, $\bar{C}_{1}$ and $\bar{C}_{2}$ are uniformly bounded, if we take $(M_{0},M_{1},M_{2},M_{3},\bar{C}(\alpha),\hat{C}(\alpha),T)\in\mathcal{M}$ with $\mathcal{M}\subset(0,\infty)^{7}$ bounded.\\[2mm]
\end{theorem}
The function $\bar{C}$ is positive and bounded in the following way because of the non-negativity of the coefficients $\alpha_{K}$: 
\begin{equation}\label{Cbar}
 0 < \zeta := \frac{\min_{1\leq K\leq N}\mu_{K}^{[2]}}{\max_{1\leq K\leq N}\kappa_{K}^{[1]}}\leq \bar{C}(\alpha) \leq \frac{\max_{1\leq K\leq N}\mu_{K}^{[2]}}{\min_{1\leq K\leq N}\kappa_{K}^{[1]}} =: \eta < \infty. 
\end{equation}
We define for the remainder of the article the spaces
\[ V:=H^{1}(\Omega,\RR^3)\;\; \mbox{ and }\;\; H:=L^{2}(\Omega,\RR^3)\]
and identify $H$ with its dual space $H'$. Then we get the Gelfand triple 
\[V \subset H=H' \subset V'\]
with dense, continuous embeddings. In addition, we consider 
\[U:= H_{0}^{1}(\Omega,\RR^3)\]
and thereby 
\[U \subset V \subset H=H' \subset V' \subset U'\]
with dense, continuous embeddings and $U'\simeq H^{-1} (\Omega,\RR^3)$.\\
We collect some further results which are useful when proving the main achievements of the article.\\
\begin{lemma}\label{L-ungl-h2}
For $u\in H^{2}(\Omega,\RR^3)$ we have
\begin{equation}\label{ungl-h2}
  \Bigg(\int\limits_{\Omega}\big(\sum_{i=1}^{3}\sum_{j=1}^{3}|\partial_{j}u_{i}|\big)^{4}dx\Bigg)^{\frac{1}{2}} \leq 27(1+C_{\Omega})C_{SE}^{2}\|u\|_{H^{2}(\Omega,\RR^3)}^{2}.
\end{equation}
for constants $C_{\Omega}>0$, $C_{SE}>0$ not depending on $u$.\\[1ex]
\end{lemma}
\begin{proof}
Let $u\in H^{2}(\Omega,\RR^3)$. Using Poincar\'{e}'s inequality
\begin{equation}\label{poincare}   \|u\|_{L^2 (\Omega)} \leq C_{\Omega} \|\nabla u\|_{L^2 (\Omega)}\qquad \mbox{for } u\in H_0^1 (\Omega),  \end{equation}
Sobolev's embedding theorem for $\Omega\subset \RR^3$ as well as H\"older's inequality yields
\begin{eqnarray*}
 \int\limits_{\Omega}\big(\sum_{i=1}^{3}\sum_{j=1}^{3}|\partial_{j}u_{i}|\big)^{4}\,dx
 &\leq & 9^{3}\sum_{i=1}^{3}\sum_{j=1}^{3}\int\limits_{\Omega}|\partial_{j}u_{i}|^{4}\,dx
 = 9^{3}\sum_{i=1}^{3}\sum_{j=1}^{3}\|\partial_{j}u_{i}\|^{4}_{L^{4}(\Omega)}\\
 &\leq& 9^{3}C_{SE}^{4}\sum_{i=1}^{3}\sum_{j=1}^{3}\|\partial_{j}u_{i}\|^{4}_{H^{1}(\Omega)}\\
 &\leq& 9^{3}(1+C_{\Omega})^{2}C_{SE}^{4}\sum_{i=1}^{3}\sum_{j=1}^{3}\big(\sum_{|\alpha|=1}\|\partial_{\alpha}\partial_{j}u_{i}\|^{2}_{L^{2}(\Omega)}\big)^{2}\\ 
 &\leq& 9^{3}(1+C_{\Omega})^{2}C_{SE}^{4}\big(\sum_{i=1}^{3}\sum_{j=1}^{3}\sum_{|\alpha|=1}\|\partial_{\alpha}\partial_{j}u_{i}\|^{2}_{L^{2}(\Omega)}\big)^{2}\\ 
 &\leq& 9^{3}(1+C_{\Omega})^{2}C_{SE}^{4}\|u\|_{H^{2}(\Omega,\RR^3)}^{4},
\end{eqnarray*}
which proves assertion (\ref{ungl-h2}).\\[1ex]
\end{proof}
\\[2mm]
Furthermore, we need the following theorem, shown by Lions \cite{lions1971opt}, which is concerned with the solvability of a special class of linear initial value problems.\\
\begin{theorem}{(Lions)}\label{lions}
Let $\mathcal{A}(t)$, $t\in[0,T]$, be a family of operators, that map from $U$ onto $U'$ and let $a(t;v_{1},v_{2}):=\langle\mathcal{A}(t)v_{1},v_{2}\rangle_{U'\times U}$ define corresponding bilinear forms on $U$, satisfying\\
 \renewcommand{\labelenumi}{\roman{enumi})}
 \begin{enumerate}
  \item $a(t;v_{1},v_{2})$ is differentiable in $t$, $\forall v_{1},v_{2}\in U$, $t\in[0,T]$\\
  \item $a(t;v_{1},v_{2})=a(t;v_{2},v_{1})$ $\forall v_{1},v_{2}\in U$, $t\in[0,T]$\\
  \item There exist constants $\alpha>0$ and $\lambda\in\RR$ with 
      \[a(t;v,v) + \lambda\|v\|_{H}^{2}\geq\alpha\|v\|_{V}^{2}\;\;\forall v \in U,\;t\in[0,T].\]
 \end{enumerate}
 Then for every $f\in L^{2}(0,T;H)$, $v_{0}\in U$ and $v_{1}\in H$ there exists a unique $v\in L^{2}(0,T;U)$ with $v'\in L^{2}(0,T;H)$, that satisfies 
 \begin{eqnarray}\label{lions-eq}
  v'' + \mathcal{A}(t)v &=& f\;\;\mbox{ on }\;\; (0,T), \nonumber \\
  v(0) &=& v_{0}, \\
  v'(0) &=& v_{1}. \nonumber 
 \end{eqnarray}
 Furthermore, the mapping 
 \[\{f,v_{0},v_{1}\}\mapsto \{v,v'\}\]
 is continuous as a map from $L^{2}(0,T;H)\times U\times H$ onto $L^{2}(0,T;U)\times L^{2}(0,T;H)$.\\
\end{theorem}
\begin{remark}
Lions and Magenes even show in \cite{lions1972non}, that the solution $v$ is in $\mathcal{C}([0,T],U)$ and $v'\in \mathcal{C}([0,T],H)$.\\[2mm]
\end{remark}
To evaluate normal derivatives of $v$ (see (\ref{lions-eq})) it is necessary to have more regularity of $v$ in $x$. This can be obtained by requiring more regularity of $f$ as well as of the initial values $v_{0}$ and $v_{1}$. The next lemma ensues directly from an application of Theorem 30.4 in \cite{wloka} (cf. \cite{BSS15}).\\
\begin{lemma}\label{wloka}
 Suppose that $f\in H^{1}(0,T;H)$, $v_{0}\in H_{0}^{2}(\Omega,\RR^3)\subset U$ and  $v_{1}\in U$. Then the unique solution $v$ of (\ref{lions-eq}) satisfies
 \[v\in\Big(H^{1}(0,T;U)\cap H^{2}(0,T;H)\Big)\]
 as well as  
 \[v\in L^{2}(0,T;H_{0}^{2}(\Omega,\RR^3)).\]
\end{lemma}
One principal technique to prove the results of this article takes advantage of Gronwall's lemma.\\
\begin{lemma}[Gronwall's lemma]
 Let $\psi\in\mathcal{C}(0,T)$ and $b,k\in L^{1}(0,T)$ be non-negative functions. If $\psi$ satisfies
 \begin{equation*}
  \psi(\tau) \leq a + \int\limits_{0}^{\tau}b(t)\psi(t)dt + \int\limits_{0}^{\tau}k(t)\psi(t)^{p}dt
 \end{equation*}
 for all $\tau\in[0,T]$ with constants $p\in(0,1)$ and $a\geq 0$, then
 \begin{equation}\label{gronwall}
  \psi(\tau) \leq \exp\Big(\int\limits_{0}^{\tau}b(t)dt\Big)\bigg[a^{1-p}+(1-p)\int\limits_{0}^{\tau}k(t)\exp((p-1)\int\limits_{0}^{t}b(\sigma)d\sigma)dt\bigg]^{1/(1-p)}
 \end{equation}
 is valid for all $\tau\in[0,T]$.\\
\end{lemma}
A proof of this version can be found in \cite{Bainov199205}. In this article, Gronwall's lemma is applied for $a$, $b$, and $k$ being constants and $p=\frac{1}{2}$. Then we get as a consequence
of (\ref{gronwall}) the estimate
\begin{equation}\label{gronwall-const}
 \psi(\tau) \leq \bigg[\exp\big(\frac{1}{2}b\tau\big)a^{\frac{1}{2}} + \frac{k}{b}\Big(\exp\big(\frac{1}{2}b\tau\big)-1\Big)\bigg]^{2}
\end{equation}
for all $\tau\in[0,T]$.\\

%%%%%%%%%%%%%%%%%%%%%%%%%%%%%%%%%%%%%%%%%%%%%%%%%%Section 3%%%%%%%%%%%%%%%%%%%%%%%%%%%%%%%%%%%%%%%%%%%%%%%%%%%%%%%%%

\section{Fr\'{e}chet derivative of the forward operator}\label{secFrechet}

%%%%%%%%%%%%%%%%%%%%%%%%%%%%%%%%%%%%%%%%%%%%%%%%%%%%%%%%%%%%%%%%%%%%%%%%%%%%%%%%%%%%%%%%%%%%%%%%%%%%%%%%%%%%%%%%%%%%

Let the domain $\Drm (\mathcal{T})\subset \RN$ of $\mathcal{T}$ be defined by 
\begin{equation}\label{domain-T}
\Drm (\mathcal{T}) := \big\{ \alpha \in \RN :  \mbox{the IBVP (\ref{cauchy-hyper-comb}), (\ref{anfangswert1})--(\ref{randwert}) has a unique solution } u\in \mathcal{A} \big\}.
\end{equation}
%We consider the forward operator $\mathcal{T}: \Drm (\mathcal{T})\subset \RN \to L^2 (0,T;V)$ which maps for fixed $(f,u_0,u_1)$ a vector $\alpha$ to the unique solution $u$ of (\ref{cauchy-hyper-comb}), (\ref{anfangswert1})--(\ref{randwert}) in $\mathcal{A}(M_{0},M_{1},M_{2},M_{3})$.
As a first step we derive a characterization of the G\^{a}teaux derivative of $\mathcal{T}$.\\
\begin{lemma}\label{frechet-system}
Let $\alpha\in \mathrm{int} \big(\Drm(\mathcal{T})\big)$ be an interior point of $\Drm(\mathcal{T})$. The G\^{a}teaux derivative $v=\mathcal{T}'(\alpha)h$ of the solution operator $\mathcal{T}$ fulfills for $h\in \RN$ the following linear system of differential equations with homogeneous initial and boundary value conditions
\begin{equation}\label{gateaux-diff}
  \rho\ddot{v}(t,x)-\nabla\cdot [\nabla_{Y}\nabla_{Y}C_{\alpha}(x,Ju(t,x)):Jv(t,x)] = \nabla \cdot [\nabla_{Y}C_{h}(x,Ju(t,x))]
\end{equation}
for $t\in[0,T]$ and $x\in\Omega\subset\RR^{3}$,
\begin{equation}\label{vAnfangswert}
  v(0,x)=\dot{v}(0,x)=0 \qquad \mbox{for } x\in\Omega
\end{equation}
and 
\begin{equation}\label{vRandwert}
 v(t,x)=0\qquad \mbox{for } x\in\partial\Omega.
\end{equation}
Here, we used the notations
\[C_{\alpha}=\sum_{K=1}^{N}\alpha_{K}C_{K}\;\;\mbox{ respectively }\;\;C_{h}=\sum_{K=1}^{N}h_{K}C_{K}.\]
\end{lemma}
\begin{proof}%[Proof]
Let $\alpha,\; h\in\RN$ and $s>0$ sufficiently small. Then $u(\alpha+sh) := \mathcal{T}(\alpha+sh)$ denotes the solution of (\ref{cauchy-hyper-comb}), (\ref{anfangswert1})--(\ref{randwert}) with $\alpha$ replaced by $\alpha+sh$. We set further $u(\alpha) := \mathcal{T} (\alpha)$. Then we have
\begin{equation}\label{diff-1}
 \rho\big(\ddot{u}(\alpha+sh)-\ddot{u}(\alpha)\big)-\nabla\cdot [\nabla_{Y}C(x,Ju(\alpha+sh))-\nabla_{Y}C(x,Ju(\alpha))] = 0.
\end{equation}
We assume for the moment, that the following limit 
\[v:=\lim_{s\to 0^{+}}\frac{u(\alpha+sh)-u(\alpha)}{s} \]
exists. Then $v=\mathcal{T}'(\alpha)h$ is the G\^{a}teaux derivative at $\alpha$ in the direction $h$. Please note that the limit process is well-defined since we assumed $\alpha\in\Drm (\mathcal{T})$ to be non-isolated. The aim is to determine a partial differential equation which is uniquely solved by $v$. For this purpose it is necessary to restate equation (\ref{diff-1}) as
\begin{eqnarray*}
 &&\nabla_{Y}C(x,Ju(\alpha+sh))-\nabla_{Y}C(x,Ju(\alpha))\\
 &=& \sum_{j=1}^{N}(\alpha_{j}+sh_{j})\nabla_{Y}C_{j}(x,Ju(\alpha+sh))-\sum_{j=1}^{N}\alpha_{j}\nabla_{Y}C_{j}(x,Ju(\alpha))\\
 %&=& \sum_{j=1}^{N}(\alpha_{j}+sh_{j})\nabla_{Y}C_{j}(x,Ju(\alpha+sh))-\sum_{j=1}^{N}(\alpha_{j}+sh_{j})\nabla_{Y}C_{j}(x,Ju(\alpha))\\
 %   && \;\;+ \sum_{j=1}^{N}(\alpha_{j}+sh_{j})\nabla_{Y}C_{j}(x,Ju(\alpha))-\sum_{j=1}^{N}\alpha_{j}\nabla_{Y}C_{j}(x,Ju(\alpha))\\
 %&=& \sum_{j=1}^{N}(\alpha_{j}+sh_{j})[\nabla_{Y}C_{j}(x,Ju(\alpha+sh))-\nabla_{Y}C_{j}(x,Ju(\alpha))]\\ 
 %   && \;\;+ \sum_{j=1}^{N}sh_{j}\nabla_{Y}C_{j}(x,Ju(\alpha))\\
 &=& \sum_{j=1}^{N}(\alpha_{j}+sh_{j})[\nabla_{Y}C_{j}(x,Ju(\alpha+sh))-\nabla_{Y}C_{j}(x,Ju(\alpha))] + s\nabla_{Y}C_{h}(x,Ju(\alpha)).
\end{eqnarray*}
On the other hand setting
\[F(r):=\nabla_{Y}C_{j}(x,rJu(\alpha+sh)+(1-r)Ju(\alpha))\]
for $r\in[0,1]$ yields
\begin{eqnarray*}
 &&\nabla_{Y}C_{j}(x,Ju(\alpha+sh))-\nabla_{Y}C_{j}(x,Ju(\alpha))
 %&=& F(1) - F(0)\\
 = \int\limits_{0}^{1}F'(r)dr\\
 &=& \int\limits_{0}^{1}\langle\langle\underbrace{\nabla_{Y}\nabla_{Y}C_{j}(x,rJu(\alpha+sh)+(1-r)Ju(\alpha))}_{:=Y_{r}},Ju(\alpha+sh)-Ju(\alpha)\rangle\rangle ds.\\
\end{eqnarray*}
Using the fact that
\[   \lim_{r\to 0^+} Y_{r} = \nabla_{Y}\nabla_{Y}C_{j}(x,Ju(\alpha)) , \]
which follows from the continuity of $\alpha \mapsto u(\alpha)$ as stated in Theorem \ref{theorem21imp}, we get
\begin{eqnarray*}
 &&\lim_{s\to 0^+} \frac{1}{s} \Big( \nabla_{Y}C(x,Ju(\alpha+sh))-\nabla_{Y}C(x,Ju(\alpha)) \Big) \\
 &=&\lim_{s\to 0^+}\sum_{j=1}^{N}(\alpha_{j}+sh_{j})\int\limits_{0}^{1}\langle\langle Y_{r},\frac{1}{s}(Ju(\alpha+sh)-Ju(\alpha))\rangle\rangle ds + \nabla_{Y}C_{h}(x,Ju(\alpha))\\
 &=& \sum_{j=1}^{N}\alpha_{j}\int\limits_{0}^{1}\langle\langle\nabla_{Y}\nabla_{Y}C_{j}(x,Ju(\alpha)),Jv\rangle\rangle ds + \nabla_{Y}C_{h}(x,Ju(\alpha))\\
 &=&\nabla_{Y}\nabla_{Y}C_{\alpha}(x,Ju(\alpha)):Jv + \nabla_{Y}C_{h}(x,Ju(\alpha)).
\end{eqnarray*}
If we insert this in equation (\ref{diff-1}) and use the definition of $v$, we obtain the differential equation (\ref{gateaux-diff}) for $v$. The boundary condition for $v$ follows directly from the definition of $v$ and (\ref{randwert}). 
The initial conditions for $v$ are finally obtained from the assumptions $u_{0}(\alpha+sh)=u_{0}(\alpha)=u_0$ and $u_{1}(\alpha+sh)=u_1 (\alpha)=u_1$.\\
\end{proof}
\\[2mm]
It remains to prove that the IBVP (\ref{gateaux-diff})--(\ref{vRandwert}) has a unique solution and hence the G\^{a}teaux derivative is well defined.\\
\begin{theorem}\label{gateaux-ex}
 The IBVP (\ref{gateaux-diff})--(\ref{vRandwert}) has a unique, weak solution $v=\mathcal{T}'(\alpha)h$ in $L^{2}(0,T;V)$.\\
\end{theorem}
\begin{proof}
We want to apply Theorem \ref{lions} to prove existence and uniqueness of a solution. To this end we note, that (\ref{gateaux-diff}) is linear in $v$, and define for $v\in U$ 
 \begin{equation}\label{A(t)}
  \mathcal{A}(t)v := - \nabla\cdot [\nabla_{Y}\nabla_{Y}C(x,Ju(\alpha)):Jv]
 \end{equation}
 as well as
 \begin{equation}\label{flions}
  f := \nabla\cdot [\nabla_{Y} C_{h}(x,Ju(\alpha))].
 \end{equation}
 In that sense we can reformulate (\ref{gateaux-diff}) as 
 \begin{equation}\label{lions-diff}
  \rho\ddot{v} + \mathcal{A}(t)v = f. 
 \end{equation}
 According to our premises, there is $\rho>0$ with (\ref{massendichte}), so that equation (\ref{gateaux-diff}) together with $\bar{\mathcal{A}}(t)=\frac{1}{\rho}\mathcal{A}(t)$ and $\bar{f}=\frac{1}{\rho}f$ can be expressed as
 \begin{equation}\label{lions-diff-bar}
  \ddot{v} + \bar{\mathcal{A}}(t)v = \bar{f}\;\;\mbox{on}\;\;[0,T]\times\Omega 
 \end{equation}
 with $v(0,x)=\dot{v}(0,x)=0$ for $x\in\Omega$ and $v(t,x)=0$ for $t\in[0,T]$ and $x\in\partial\Omega$. 
 To apply Theorem \ref{lions} to (\ref{lions-diff-bar}), we show at first, that the mapping $\bar{\mathcal{A}}(t): U\rightarrow U'$ is linear and continuous on $U$. 
 The linearity is obvious. It remains to show that $\bar{\mathcal{A}}(t)v\in U'$ for all $v\in U$.
 For this purpose we define 
 \begin{equation*}
  a(t;w,v) := \langle \bar{\mathcal{A}}(t)w,v\rangle_{U'\times U} := \int\limits_\Omega \langle - \nabla\cdot[\nabla_{Y}\nabla_{Y}C(x,Ju(\alpha)):Jw],v\rangle\,d x
 \end{equation*}
 for $v,w\in U$. Using the Gau{\ss}-Ostrogradski theorem yields 
 \begin{equation}\label{a(w,v)}
 a(t;w,v) = \int\limits_{\Omega}\langle\langle \nabla_{Y}\nabla_{Y}C(x,Ju(\alpha)):Jw,Jv\rangle\rangle \, dx. 
 \end{equation}
 From (\ref{a(w,v)}) we see, that $a(t;w,v)$ is linear in $w\in U$ and $v\in U$ for fixed $t\in[0,T]$, and thus defines a bilinear form on $U \times U$. 
 The H\"{o}lder inequality and (\ref{bed2}) imply for $w,v\in U$
 \begin{eqnarray*}
  |a(t;w,v)|
  &=& \Big|\int\limits_{\Omega}\langle\langle \nabla_{Y}\nabla_{Y}C(x,Ju(\alpha)):Jw,Jv\rangle\rangle dx \Big|\\
  && \leq \sum_{i,j,k,l=1}^{3}\sum_{K=1}^{N}\alpha_{K}\int\limits_{\Omega} \big|\partial_{Y_{ij}}\partial_{Y_{kl}}C_{K}(x,Ju(\alpha))\big| \big|\partial_{l}w_{k} \big| \big|\partial_{j}v_{i} \big|dx\\
  %&& \leq \sum_{i,j,k,l=1}^{3}\sum_{K=1}^{N}\alpha_{K}\mu_{K}^{[1]}\int\limits_{\Omega}|\partial_{l}w_{k}||\partial_{j}v_{i}|dx\\
  %&& \leq \sum_{K=1}^{N}\alpha_{K}\mu_{K}^{[1]}\big(\int\limits_{\Omega}(\sum_{k,l=1}^{3}|\partial_{l}w_{k}|)^{2}dx\big)^{\frac{1}{2}}\big(\int\limits_{\Omega}(\sum_{i,j=1}^{3}|\partial_{j}v_{i}|)^{2}dx\big)^{\frac{1}{2}}\\
  %&& \leq 9\sum_{K=1}^{N}\alpha_{K}\mu_{K}^{[1]}\big(\int\limits_{\Omega}\sum_{k,l=1}^{3}|\partial_{l}w_{k}|^{2}dx\big)^{\frac{1}{2}}\big(\int\limits_{\Omega}\sum_{i,j=1}^{3}|\partial_{j}v_{i}|^{2}dx\big)^{\frac{1}{2}}\\
  && \leq 9\underbrace{\sum_{K=1}^{N}\alpha_{K}\mu_{K}^{[1]}}_{>0}\|w\|_{U}\|v\|_{U}.
 \end{eqnarray*}
 %and hence
 %\[ |a(t;w,v)| \leq \underbrace{9\sum_{K=1}^{N}\alpha_{K}\mu_{K}^{[1]}}_{>0}\|w\|_{U}\|v\|_{U}.\]
 This gives the continuity of $a(t;w,v)$ in $w,v\in U$ and particularly $\bar{\mathcal{A}}(t)v\in U'$ for $v\in U$.\\
 Next we check that conditions i)--iii) of Theorem \ref{lions} to be fulfilled. The assumptions that $Y\mapsto C_{K}(x,Y)$ is three times continuously differentiable for $x\in\Omega$ almost everywhere
 and $u\in\mathcal{A}$ imply that $a(t;w,v)$ is differentiable with respect to $t\in [0,T]$ for $w,v\in U$ validating i).
 Using the condition 
 \begin{equation}\label{symm}
  \partial_{Y_{ij}}\partial_{Y_{kl}}C_{K}(x,Y)=\partial_{Y_{kl}}\partial_{Y_{ij}}C_{K}(x,Y)
 \end{equation}
 for all $i,j,k,l=1,2,3$ and $Y\in\RR^{3\times3}$ and the Gau{\ss}-Ostrogradski theorem, we get the symmetry of $a(t;w,v)$. Thus ii) holds true. 
 %in the following way: 
 %\begin{eqnarray*}
  %&& a(t;w,v)\\
  %&& = \int\limits_{\Omega}\langle\langle \nabla_{Y}\nabla_{Y}C(x,Ju(\alpha)):Jw,Jv\rangle\rangle dx\\
  %&& = \sum_{i,j,k,l=1}^{3}\sum_{K=1}^{N}\alpha_{K}\int\limits_{\Omega}\partial_{Y_{ij}}\partial_{Y_{kl}}C_{K}(x,Ju(\alpha))\partial_{l}w_{k}\partial_{j}v_{i}dx\\
  %&& = \sum_{i,j,k,l=1}^{3}\sum_{K=1}^{N}\alpha_{K}\int\limits_{\Omega}\partial_{Y_{kl}}\partial_{Y_{ij}}C_{K}(x,Ju(\alpha))\partial_{j}v_{i}\partial_{l}w_{k}dx\\
  %&& = \int\limits_{\Omega}\langle\langle \nabla_{Y}\nabla_{Y}C(x,Ju(\alpha)):Jv,Jw\rangle\rangle dx\\
  %&& = a(t;v,w).\\
  %\end{eqnarray*}
  It remains to prove iii). Let $v\in U$ and $t\in[0,T]$. Applying (\ref{bed2}) we obtain 
  \begin{eqnarray*}
  a(t;v,v)
  &=& \int\limits_{\Omega}\langle\langle \nabla_{Y}\nabla_{Y}C(x,Ju(\alpha)):Jv,Jv\rangle\rangle dx \\
  && \geq \sum_{K=1}^{N}\alpha_{K}\kappa_{K}^{[1]}\int\limits_{\Omega}\|Jv\|_{F}^{2}dx
  %&& = \sum_{K=1}^{N}\alpha_{K}\kappa_{K}^{[1]}\sum_{i,j=1}^{3}\int\limits_{\Omega}|\partial_{j}v_{i}|^{2}dx\\
  = \gamma \, \|v\|_{U}^{2}
 \end{eqnarray*}
 %and for this reason
 %\begin{equation}\label{a(v,v)}
 % a(t;v,v) \geq \underbrace{\sum_{K=1}^{N}\alpha_{K}\kappa_{K}^{[1]}}_{=:\gamma>0}\|v\|_{U}^{2}.
 %\end{equation}
 with $\gamma := \sum_{K=1}^N \alpha_K \kappa_K^{[1]}>0$. If we use this estimate together with the definition of the spaces $U$, $V$ and $H$, we get
 \[a(t;v,v) + \gamma\|v\|_{H}^{2}\geq \gamma\|v\|_{U}^{2}\qquad \mbox{for all } v\in U .  \]
 This gives iii) with $\lambda=\alpha=\gamma$. Thus, all requirements of Theorem \ref{lions} are satisfied, which proves that
 there is for all $\bar{f}\in L^{2}(0,T;H)$ a unique solution $v\in L^{2}(0,T;U)$ of (\ref{lions-diff-bar}) with $\dot{v}\in L^{2}(0,T;H)$ and hence the IBVP (\ref{gateaux-diff})--(\ref{vRandwert})
 has a unique solution $v\in L^{2}(0,T;U)$, too. Because of the continuous embedding of $U$ in $V$ we even have $v\in L^{2}(0,T;V)$, what completes the proof.\\ 
 \end{proof}
 \\[1mm]
 \begin{remark}
 There exists a unique solution $v$ of the IBVP (\ref{gateaux-diff})--(\ref{vRandwert}) for all $h\in\RR^N$ and hence $\mathcal{T}' (\alpha) h$ is well defined on the linear space $\RR^N$
 for any non-isolated $\alpha\in \Drm (\mathcal{T})$. This is important in view of proving that $v=\mathcal{T}' (\alpha) h$ represents the Fr\'{e}chet derivative.\\[1ex]
 \end{remark}
 Our aim is to show, that $\mathcal{T}: \Drm (\mathcal{T})\subset \RN\to L^{2}(0,T;V)$ even is Fr\'{e}chet differentiable, where the Fr\'{e}chet derivative is represented by $v=\mathcal{T}'(\alpha)h$ 
 solving (\ref{gateaux-diff})--(\ref{vRandwert}). To this end our next step is to prove that the mapping $\mathcal{T}' (\alpha) : \RR^N\to  L^{2}(0,T;V)$, $h\mapsto v$, is linear and bounded.
 The linearity follows immediately from the fact, that (\ref{gateaux-diff}) is linear in $v$ and $h$. The continuity is subject of the following theorem.\\
 \begin{theorem}\label{gateaux-stetig}
 Adopt the assumption of Lemma \ref{frechet-system}. The G\^{a}teaux derivative $v=\mathcal{T}'(\alpha)h$ is continuous in $h$ for all $h\in\RR^{N}$, i.e. there is a constant $L_{1}>0$ with $\|v\|_{L^{2}(0,T;V)}\leq L_{1}\|h\|_{\infty}$.\\
 \end{theorem}
 \begin{proof}
 Multiplying equation (\ref{gateaux-diff}) by $2\dot{v}$ and integrating over $\Omega$ yield
 \begin{eqnarray*}
 && 2\langle\rho\ddot{v}(t,\cdot),\dot{v}(t,\cdot)\rangle_{L^{2}(\Omega,\RR^3)}
 -2\langle\nabla \cdot [\nabla_{Y}\nabla_{Y}C_{\alpha}(x,Ju(t,x)):Jv(t,\cdot)],\dot{v}(t,\cdot)\rangle_{L^{2}(\Omega,\RR^3)}\\[1ex]
 && = 2\langle\nabla\cdot [\nabla_{Y}C_{h}(x,Ju(t,x))],\dot{v}(t,\cdot)\rangle_{L^{2}(\Omega,\RR^3)}
 \end{eqnarray*}
 for $t\in[0,T]$. Since $u=\mathcal{T} (\alpha ) \in \mathcal{A}$, we have that $v\in L^{2}(0,T;H_{0}^{2}(\Omega,\RR^3))$ (see Lemma \ref{wloka}). 
 Using the Gau{\ss}-Ostrogradski theorem we get 
 \begin{eqnarray}\label{vschwach}
  &&2\sum_{K=1}^{N}\alpha_{K}\int\limits_{\Omega}\langle\langle\nabla_{Y}\nabla_{Y}C_{K}(x,Ju(t,x)):Jv(t,x),J\dot{v}(t,x)\rangle\rangle dx
  +2\langle\rho\ddot{v}(t,\cdot),\dot{v}(t,\cdot)\rangle_{L^{2}(\Omega,\RR^3)} \nonumber \\[1ex]
  &&= 2\langle\nabla\cdot [\nabla_{Y}C_{h}(x,Ju(t,x))],\dot{v}(t,\cdot)\rangle_{L^{2}(\Omega,\RR^3)}.
 \end{eqnarray}
 We define 
 \[a_{1}(t;v,\dot{v}):=\sum_{K=1}^{N}\alpha_{K}\int\limits_{\Omega}\langle\langle\nabla_{Y}\nabla_{Y}C_{K}(x,Ju(t,x)):Jv(t,x),J\dot{v}(t,x)\rangle\rangle\, dx\]
 and 
 \[f_{1}:=\nabla\cdot [\nabla_{Y}C_{h}(x,Ju(t,x))].\]
 The bilinear form $a_{1}(t;\cdot,\cdot)$ is symmetric on $V\times V$ for all $t\in[0,T]$ because of 
  \[\partial_{Y_{ij}}\partial_{Y_{kl}}C_{K}(x,Y)=\partial_{Y_{kl}}\partial_{Y_{ij}}C_{K}(x,Y)\]
 for all $i,j,k,l=1,2,3$ and $Y\in\RR^{3\times3}$ (see proof of Theorem \ref{gateaux-ex}), whence
 \[2a_{1}(t;v,\dot{v}) + \rho\partial_{t}\|\dot{v}(t,\cdot)\|_{L^{2}(\Omega,\RR^3)}^{2} = 2\langle f_{1}(t,\cdot), \dot{v}(t,\cdot)\rangle_{L^{2}(\Omega,\RR^3)}\]
 and thereby
 \[\partial_{t}[a_{1}(t;v,v) + \rho\|\dot{v}(t,\cdot)\|_{L^{2}(\Omega,\RR^3)}^{2}] = a'_{1}(t;v,v) + 2\langle f_{1}(t,\cdot), \dot{v}(t,\cdot)\rangle_{L^{2}(\Omega,\RR^3)} \]
 follows. Integrating the last equation over $[0,\tau]$ with $0\leq\tau\leq T$ yields together with $v(0,x)=\dot{v}(0,x)=0$
 \begin{equation}\label{gleichung-v}
  a_{1}(t;v,v) + \rho\|\dot{v}(\tau,\cdot)\|_{L^{2}(\Omega,\RR^3)}^{2} = \int\limits_{0}^{\tau}a'_{1}(t;v,v)dt + 2\int\limits_{0}^{\tau}\langle f_{1}(t,\cdot), \dot{v}(t,\cdot)\rangle_{L^{2}(\Omega,\RR^3)}dt.
 \end{equation}
 We proceed by appropriately estimating the different summands in equation (\ref{gleichung-v}). 
 For the first one, we obtain using (\ref{bed2})
 \begin{equation*}
 a_{1}(t;v,v) \geq \sum_{K=1}^{N}\alpha_{K}\kappa_{K}^{[1]}\|Jv(\tau,\cdot)\|_{L^{2}(\Omega,\RR^{3\times3})}^{2}
 \end{equation*}
 and thus with the assumptions of Theorem \ref{theorem21imp} 
 \begin{equation}\label{a1}
 a_{1}(t;v,v) \geq \kappa(\alpha)\|Jv(\tau,\cdot)\|_{L^{2}(\Omega,\RR^{3\times3})}^{2}.
 \end{equation}
 Using the conditions (\ref{beschr1}) and (\ref{bedA}) leads to
 \begin{eqnarray*}
 a'_{1}(t;v,v)
 &=& \sum_{K=1}^{N}\alpha_{K}\int\limits_{\Omega}\langle\langle[\nabla_{Y}\nabla_{Y}\nabla_{Y}C_{K}(x,Ju(t,x)):J\dot{u}(\alpha)]:Jv(t,x),Jv(t,x)\rangle\rangle dx\\[1ex]
 &\leq& 9\sum_{K=1}^{N}\alpha_{K}\mu_{K}^{[2]}M_{1}\int\limits_{\Omega}(\sum_{i=1}^{3}\sum_{j=1}^{3}|\partial_{j}v_{i}|)^{2}dx
\end{eqnarray*}
 and we get with an application of H\"{o}lder's inequality
\begin{equation*}
 a'_{1}(t;v,v) \leq 81\sum_{K=1}^{N}\alpha_{K}\mu_{K}^{[2]}M_{1}\|Jv(t,\cdot)\|_{L^{2}(\Omega,\RR^{3\times3})}^{2}.
\end{equation*}
Finally Theorem \ref{theorem21imp} and (\ref{Cbar}) show
\begin{equation}\label{a1'}
 a'_{1}(t;v,v) \leq \frac{729}{8}\eta\mu(\alpha)M_{1}\|Jv(t,\cdot)\|_{L^{2}(\Omega,\RR^{3\times3})}^{2}.
\end{equation}
For the last term of equation (\ref{gleichung-v}) we estimate by means of (\ref{bed2}), (\ref{beschr3}), (\ref{bedA}) and several applications of the H\"{o}lder's inequality
\begin{eqnarray*}
 \langle f_{1}(t,\cdot), \dot{v}(t,\cdot)\rangle_{L^{2}(\Omega,\RR^3)}
 %&=& \langle \nabla\cdot[\nabla_{Y}C_{h}(x,Ju(t,x))], \dot{v}(t,\cdot)\rangle_{L^{2}(\Omega,\RR^3)} \\
 &=& \int\limits_{\Omega}\sum_{k=1}^{3}\sum_{l=1}^{3}\partial_{l}[\partial_{Y_{kl}}C_{h}(x,Ju(\alpha))]\dot{v}_{k}(t,x)\,dx\\[1ex]
 &&\hspace*{-3.5cm}\leq 3\sum_{K=1}^{N}|h_K|\mu_{K}^{[4]}\sum_{k=1}^{3}\int\limits_{\Omega}|\dot{v}_{k}(t,x)|\,dx
 +\sum_{K=1}^{N}|h_K|\mu_{K}^{[1]}\sum_{i,j,k,l=1}^{3}\int\limits_{\Omega}|\dot{v}_{k}(t,x)||\partial_{l}\partial_{j}u_{i}(\alpha)|\,dx\\[1ex]
 &&\hspace*{-3.5cm}\leq 3\sqrt{3\mbox{vol}(\Omega)}\sum_{K=1}^{N}|h_K|\mu_{K}^{[4]}\|\dot{v}(t,\cdot)\|_{L^{2}(\Omega,\RR^3)} + 27M_{0}\sum_{K=1}^{N}|h_K|\mu_{K}^{[1]}\|\dot{v}(t,\cdot)\|_{L^{2}(\Omega,\RR^3)}.
\end{eqnarray*}
Using (\ref{massendichte}) we obtain
\begin{equation}\label{fvpunkt}
 \langle f_{1}(t,\cdot), \dot{v}(t,\cdot)\rangle_{L^{2}(\Omega,\RR^3)} \leq 3\rho_{\min}^{-\frac{1}{2}}\sum_{K=1}^{N}|h_K|
\Big( \sqrt{3\mbox{vol}(\Omega)}\mu_{K}^{[4]}+9M_{0}\mu_{K}^{[1]} \Big) \rho^{\frac{1}{2}}\|\dot{v}(t,\cdot)\|_{L^{2}(\Omega,\RR^3)}.
\end{equation}
Putting all these estimates (\ref{gleichung-v}), (\ref{a1}), (\ref{a1'}), (\ref{fvpunkt}) and (\ref{Cbar}) together and rearranging terms, we finally arrive with 
\[S_{1}:=\rho_{\min}^{-\frac{1}{2}}\sum_{K=1}^{N}(\sqrt{3\mbox{vol}(\Omega)}\mu_{K}^{[4]}+9M_{0}\mu_{K}^{[1]})\]
at 
\begin{eqnarray*}
 && \kappa(\alpha)\|Jv(\tau,\cdot)\|_{L^{2}(\Omega,\RR^{3\times3})}^{2} + \rho\|\dot{v}(\tau,\cdot)\|_{L^{2}(\Omega,\RR^3)}^{2}\\
 &\leq& \frac{729\mu(\alpha)}{8\kappa(\alpha)}\eta M_{1}\int\limits_{0}^{\tau}\kappa(\alpha)\|Jv(t,\cdot)\|_{L^{2}(\Omega,\RR^{3\times3})}^{2} + \rho\|\dot{v}(t,\cdot)\|_{L^{2}(\Omega,\RR^3)}^{2}dt +\\
 &&+ 6\|h\|_{\infty}S_{1}\int\limits_{0}^{\tau}\bigg(\kappa(\alpha)\|Jv(t,\cdot)\|_{L^{2}(\Omega,\RR^{3\times3})}^{2} + \rho\|\dot{v}(t,\cdot)\|_{L^{2}(\Omega,\RR^3)}^{2}\bigg)^{\frac{1}{2}}dt.
\end{eqnarray*}
Setting in (\ref{gronwall-const}) $a=0$, $p=\frac{1}{2}$, 
\[\psi(\tau) = \kappa(\alpha)\|Jv(\tau,\cdot)\|_{L^{2}(\Omega,\RR^{3\times3})}^{2} + \rho\|\dot{v}(\tau,\cdot)\|_{L^{2}(\Omega,\RR^3)}^{2},\]
\[b=\frac{729\mu(\alpha)}{8\kappa(\alpha)}\eta M_{1}\]
and
\[k= 6\|h\|_{\infty}S_{1},\]
we get for $\tau\in[0,T]$ 
\begin{equation}
 \kappa(\alpha)\|Jv(\tau,\cdot)\|_{L^{2}(\Omega,\RR^{3\times3})}^{2} + \rho\|\dot{v}(\tau,\cdot)\|_{L^{2}(\Omega,\RR^3)}^{2}
 \leq \frac{k^{2}}{b^{2}}\big(\exp(\frac{1}{2}b\tau)-1\big)^{2}
\end{equation}
and hence
\begin{equation}\label{vUungl}
 \|v(\tau,\cdot)\|_{U} \leq \bar{L}_{1}S(\tau)\|h\|_{\infty}
\end{equation}
with 
\[\bar{L}_{1}:=\frac{16S_{1}\sqrt{\kappa(\alpha)}}{243\mu(\alpha)\eta M_{1}}\]
and
\begin{equation}\label{S-tau} S(\tau):=\exp\Big(\frac{729\mu(\alpha)}{16\kappa(\alpha)}\eta M_{1}\tau\Big)-1. \end{equation}
Using the fact, that $S(\tau)$ is monotonically increasing for $\tau\in[0,T]$, and the mean value theorem, we derive 
\begin{equation*}
 \|v\|_{L^{2}(0,T;U)}\leq \bar{L}_{1}S(T)T\|h\|_{\infty}
\end{equation*}
and with Poincar\'{e}'s inequality (\ref{poincare})
\begin{equation}\label{vungl}
 \|v\|_{L^{2}(0,T;V)}\leq L_{1}\|h\|_{\infty},
\end{equation}
where
\[L_{1}:=\sqrt{1+C_{\Omega}}\bar{L}_{1}\Big(\exp\Big(\frac{729\mu(\alpha)}{16\kappa(\alpha)}\eta M_{1}T\Big)-1\Big)T>0.\]
This completes the proof.\\
\end{proof}
\\[2mm]
To prove the Fr\'{e}chet differentiability of $\mathcal{T}$ it remains to show the convergence
\begin{equation}\label{glm-konv}
 \lim_{h\to0}\frac{\|u(\alpha+h)-u(\alpha)-v\|_{L^{2}(0,T;V)}}{\|h\|_{\infty}}=0\qquad \mbox{for } v=\mathcal{T}' (\alpha) h.
\end{equation}
To prove this we need a further, technical result.\\
\begin{lemma}
 Let $t\in[0,T]$ and $x\in\Omega$. Then $\tilde{u}=u(\alpha+h)-u(\alpha)$ satisifies
 \begin{equation}\label{ungl4-2}
  \bigg(\sum_{k,l=1}^{3}\int\limits_{\Omega}|\partial_{l}\dot{\tilde{u}}_{k}|^{4}dx\bigg)^{\frac{1}{4}} \leq \sqrt{2M_{1}}\|J\dot{\tilde{u}}(t,\cdot)\|_{L^{2}(\Omega,\RR^{3\times 3})}^{\frac{1}{2}},
 \end{equation}
 provided that $u(\alpha+h),u(\alpha)\in\mathcal{A}$.\\[1ex]
\end{lemma}
\begin{proof}
 Using the triangle inequality and (\ref{bedA}) for $k,l=1,2,3$ with $M_{1}>0$ (see Theorem \ref{theorem21imp}), we can show 
 \[\frac{|\partial_{l}\dot{\tilde{u}}_{k}|}{2M_{1}}\leq1\]
 and thusly
 \begin{equation}\label{beweisungl4-2}
  \bigg(\frac{|\partial_{l}\dot{\tilde{u}}_{k}|}{2M_{1}}\bigg)^{4}\leq\bigg(\frac{|\partial_{l}\dot{\tilde{u}}_{k}|}{2M_{1}}\bigg)^{2}.
 \end{equation}
 Using this we get 
 \begin{eqnarray*}
  \bigg(\sum_{k,l=1}^{3}\int\limits_{\Omega}|\partial_{l}\dot{\tilde{u}}_{k}|^{4}dx\bigg)^{\frac{1}{4}}
  &=&\bigg((2M_{1})^{4}\sum_{k,l=1}^{3}\int\limits_{\Omega}\bigg(\frac{|\partial_{l}\dot{\tilde{u}}_{k}|}{2M_{1}}\bigg)^{4}dx\bigg)^{\frac{1}{4}}\\
  &\leq & \bigg((2M_{1})^{4}\sum_{k,l=1}^{3}\int\limits_{\Omega}\bigg(\frac{|\partial_{l}\dot{\tilde{u}}_{k}|}{2M_{1}}\bigg)^{2}dx\bigg)^{\frac{1}{4}}\\
  %&&=\bigg((2M_{1})^{2}\sum_{k,l=1}^{3}\int\limits_{\Omega}|\partial_{l}\dot{\tilde{u}}_{k}|^{2}dx\bigg)^{\frac{1}{4}}\\
  &=& \sqrt{2M_{1}}\|J\dot{\tilde{u}}(t,\cdot)\|_{L^{2}(\Omega,\RR^{3\times 3})}^{\frac{1}{2}}.
 \end{eqnarray*}
\end{proof}
\\[2mm]
We are now able to formulate the main result of the article whose proof is postponed to the appendix.\\
\begin{theorem}\label{thm-glm-konv}
Adopt the assumptions of Lemma \ref{frechet-system} and let $v=\mathcal{T}' (\alpha)h$. There is a constant $L_{2}>0$ depending only on $\Omega$, $T$ and $\alpha$ such that 
\begin{equation}\label{glm-konv-ungl}
  \|u(\alpha+h)-u(\alpha)-v\|_{L^{2}(0,T;V)}\leq L_{2}\|h\|_{\infty}^{\frac{3}{2}}\qquad \mbox{for }\|h\|_{\infty}\rightarrow 0.\vspace{2mm}
\end{equation}
\end{theorem}
\\
This shows, that $\mathcal{T} : \RN \to L^2 (0,T;V)$ in fact is Fr\'{e}chet differentiable, where the Fr\'{e}chet derivative $\mathcal{T}'(\alpha) : \RR^N\to  L^2 (0,T;V)$ is
defined as linear and bounded mapping by $v:=\mathcal{T}'(\alpha)h$, $h\in\RR^N$, and $v$ is
characterized by the unique (weak) solution of the IBVP (\ref{gateaux-diff})--(\ref{vRandwert}).

%%%%%%%%%%%%%%%%%%%%%%%%%%%%%%%%%%%%%%%%%%%%%%%%%%%%%%%%%%%%%%%%%%%%%%%%%%%%%%%%%%%%%%%%%%%%%%%%%%%

\section{The adjoint operator $\mathcal{T}'(\alpha)^*$ of the Fr\'{e}chet derivative}

%%%%%%%%%%%%%%%%%%%%%%%%%%%%%%%%%%%%%%%%%%%%%%%%%%%%%%%%%%%%%%%%%%%%%%%%%%%%%%%%%%%%%%%%%%%%%%%%%%%%

In the last section we have shown, that the parameter-to-solution operator $\mathcal{T}$ is Fr\'{e}chet differentiable. In the following we want to determine the adjoint operator 
of the Fr\'{e}chet derivative $\mathcal{T}'(\alpha)^*$. The adjoint is important when applying iterative solvers as the Landweber method or Newton-type methods to the inverse
problem $\mathcal{T} (\alpha) = u^{meas}$ or related problems.\\
According to Lemma \ref{frechet-system} $v=\mathcal{T}'(\alpha)h$ solves the system of linear differential equations (\ref{gateaux-diff}) with homogeneous initial and boundary
values (\ref{vAnfangswert}), (\ref{vRandwert}). Let us consider the hyperbolic equation (\ref{gateaux-diff}) with arbitrary right-hand side $f\in L^2 (0,T;H)$,
\begin{equation}\label{gateaux-diff-2}
\rho\ddot{v}(t,x)-\nabla\cdot [\nabla_{Y}\nabla_{Y}C_{\alpha}(x,Ju(t,x)):Jv(t,x)] = f.
\end{equation}
Theorem \ref{gateaux-ex} says, that together with (\ref{vAnfangswert}), (\ref{vRandwert}) this system has a unique solution $v$ in $L^{2}(0,T;V)$ 
which is even in $L^{2}(0,T;U)$. Let the initial and boundary values
(\ref{vAnfangswert}), (\ref{vRandwert}) be fixed. Then we define by $\mathcal{X}$ the space, which consists of all solutions of (\ref{gateaux-diff-2}) for $f\in L^2 (0,T;H)$, i.e. if we define
the mapping $B: L^{2}(0,T;U) \to H^{-1} (0,T;H)$ by
\begin{equation}\label{Bv}
 Bv := \rho\ddot{v} - \nabla\cdot [\nabla_{Y}\nabla_{Y}C_{\alpha}(x,Ju):Jv],
\end{equation}
then $\mathcal{X} = B^{-1} (L^2 (0,T;H))\subset L^{2}(0,T;U)$. The mapping $B: \mathcal{X}\to L^2 (0,T;H)$ is then bijective
since (\ref{gateaux-diff-2}) with (\ref{vAnfangswert}), (\ref{vRandwert}) is uniquely solvable; we have $B^{-1} : L^2 (0,T;H)\to \mathcal{X}$, $f\mapsto v$, where $v$ solves (\ref{gateaux-diff-2}) with (\ref{vAnfangswert}), (\ref{vRandwert}).
%\[    \mathcal{X} := \big\{ v\in  L^{2}(0,T;U) : v \mbox{ solves } (\ref{gateaux-diff}) \mbox{ for } f\in L^2 (0,T;H) \mbox{ with initial boundary values }  (\ref{vAnfangswert}), (\ref{vRandwert}) \big\}
%      \subset L^{2}(0,T;U).  \]
Endowed with the norm $\|v\|_{\mathcal{X}}=\|B v\|_{L^{2}(0,T;H)}$, $\mathcal{X}$ turns into a Hilbert space, which is a closed subspace of $L^{2}(0,T;U)$. The embedding $\mathcal{X} \hookrightarrow L^{2}(0,T;U)$ even
is continuous.\\
\begin{lemma}
The embedding $\mathcal{X} \hookrightarrow L^{2}(0,T;U)$ is continuous, i.e. there is a constant $C>0$ not depending on $v$ such that
\[ \|v\|_{L^{2}(0,T;U)}\leq C\|v\|_{\mathcal{X}}\,,\qquad v\in \mathcal{X}\,.    \]
\end{lemma}
\begin{proof}
Theorem \ref{lions} states that the map $\bar{f} \mapsto v$ with $\bar{f}=\frac{1}{\rho}f$ and $v\in \mathcal{X}\subset L^{2}(0,T;U)$ is continuous as a map from $L^{2}(0,T;H)$ to $L^{2}(0,T;U)$.
Hence there is a constant $C>0$ independent from $v$ such that
\[    \|v\|_{L^{2}(0,T;U)}\leq C \|\bar{f}\|_{L^{2}(0,T;H)} = C \|BB^{-1}\bar{f}\|_{L^{2}(0,T;H)} = C \|v\|_{\mathcal{X}}   \]
and the assertion is proved.\\
\end{proof}
\\[2mm]
The next theorem gives a representation of the adjoint operator $\mathcal{T}'(\alpha)^* : \mathcal{X}' \to \RR^N$ for fixed $\alpha\in \RN$.\\
\begin{theorem}\label{Tadjungiert}
Let $\alpha\in \mathrm{int} \Drm (\mathcal{T})\subset \RN$ be fixed and $w\in \mathcal{X}'$.
The adjoint operator of the Fr\'{e}chet derivative $\mathcal{T}' (\alpha) : \RR^N \to \mathcal{X}$ is given by 
\begin{equation}
 \mathcal{T}'(\alpha)^{*}w = \Big[-\int\limits_{0}^{T}\int\limits_{\Omega}\langle\langle\nabla_{Y}C_{K}(x,Ju(t,x)),J(B^{-1})^{*}w(t,x)\rangle\rangle \,dx\,dt\Big]_{K=1,...,N}\in\mathbb{R}^{N},
\end{equation}
where $p:=(B^{-1})^{*}w$ is the weak solution of the hyperbolic, backward IBVP
\begin{eqnarray}
&&\hspace*{-3mm}\label{backward-1} \rho\ddot{p}(t,x)-\nabla\cdot [\nabla_{Y}\nabla_{Y}C_{\alpha}(x,Ju(t,x)):Jp(t,x)]=w(t,x)\\[1ex]
&&\hspace*{-3mm}\label{backward-2} p(T,x)=\dot{p}(T,x)=0\qquad x\in\Omega\\[1ex]
&&\hspace*{-3mm}\label{backward-3} \big(p(t,\xi)\cdot \nabla_{Y}\nabla_{Y}C_{\alpha}(\xi,Ju(t,\xi))\big)\cdot \nu (\xi)=0\qquad (t,\xi)\in[0,T]\times\partial\Omega. 
\end{eqnarray}
\end{theorem}
\begin{proof}
 Let $w\in \mathcal{X}'$. Using the Gau{\ss}-Ostrogradski theorem, we get 
 \begin{eqnarray*}
  \langle \mathcal{T}'(\alpha)h,w\rangle_{\mathcal{X}\times \mathcal{X}'}
  &=& \int\limits_{0}^{T}\int\limits_{\Omega}\langle v(t,x),w(t,x) \rangle_{\mathbb{R}^{3}}\,dx\,dt \\
  &=& \int\limits_{0}^{T}\int\limits_{\Omega}\langle B^{-1}(\nabla\cdot[\nabla_{Y}C_{h}(x,Ju(t,x))]),w(t,x) \rangle_{\mathbb{R}^{3}} \,dx\,dt \\
  &=& \int\limits_{0}^{T}\int\limits_{\Omega}\langle \nabla\cdot[\nabla_{Y}C_{h}(x,Ju(t,x))],(B^{-1})^{*}w(t,x) \rangle_{\mathbb{R}^{3}}\,dx\,dt \\
  &=& \int\limits_{0}^{T}\int\limits_{\partial\Omega}\nabla_{Y}C_{h}(\xi,Ju(t,\xi))(B^{-1})^{*}w(t,\xi)\cdot\nu(\xi)\,d\xi dt \\
  &&\;\;	-\int\limits_{0}^{T}\int\limits_{\Omega}\langle\langle \nabla_{Y}C_{h}(x,Ju(t,x)),J(B^{-1})^{*}w(t,x) \rangle\rangle \,dx\,dt \\
  %&&=\int\limits_{0}^{T}\int\limits_{\partial\Omega}\sum\limits_{K=1}^{N}h_{K}\nabla_{Y}C_{K}(\xi,Ju(t,\xi))(B^{-1})^{*}w(t,\xi)\cdot\nu(\xi)d\xi dt \\
  %&&\;\;	-\int\limits_{0}^{T}\int\limits_{\Omega}\sum\limits_{K=1}^{N}h_{k}\langle\langle \nabla_{Y}C_{K}(x,Ju(t,x)),J(B^{-1})^{*}w(t,x) \rangle\rangle dxdt \\
  &=& \bigg\langle h, \bigg[\int\limits_{0}^{T}\bigg(\int\limits_{\partial\Omega}\nabla_{Y}C_{K}(\xi,Ju(t,\xi))(B^{-1})^{*}w(t,\xi)\cdot\nu(\xi)\,d\xi \\
  &&\;\;        - \int\limits_{\Omega}\langle\langle \nabla_{Y}C_{K}(x,Ju(t,x)),J(B^{-1})^{*}w(t,x) \rangle\rangle dx\bigg)dt\bigg]_{K=1,...,N}\bigg\rangle_{\mathbb{R}^{N}}
 \end{eqnarray*}
 and hence
 \begin{eqnarray}\label{T'1}
  \mathcal{T}'(\alpha)^{*}w &=& \bigg[\int\limits_{0}^{T}\bigg(\int\limits_{\partial\Omega}\nabla_{Y}C_{K}(\xi,Ju(t,\xi))(B^{-1})^{*}w(t,\xi)\cdot\nu(\xi)d\xi \\[1ex]
  &&  - \int\limits_{\Omega}\langle\langle \nabla_{Y}C_{K}(x,Ju(t,x)),J(B^{-1})^{*}w(t,x) \rangle\rangle dx\bigg)dt\bigg]_{K=1,...,N}. \nonumber
 \end{eqnarray}
 Let for $w\in \mathcal{X}'$ and $p\in L^{2}(0,T;H)$ the adjoint $(B^{-1})^{*}:\mathcal{X}'\rightarrow L^{2}(0,T;H)$ be given by $w\mapsto p=(B^{-1})^{*}w$. 
 Let furthermore $v\in \mathcal{X}$ be a solution of (\ref{gateaux-diff}) with (\ref{vAnfangswert}) and (\ref{vRandwert}). Then we have 
 \[
 \langle v,w\rangle_{\mathcal{X}\times \mathcal{X}'}
 =\langle B^{-1}(f_{1}),w\rangle_{\mathcal{X}\times \mathcal{X}'}
 =\langle f_{1},(B^{-1})^{*}w\rangle_{L^{2}(0,T;H)}
 =\langle Bv,p\rangle_{L^{2}(0,T;H)}.
 \]  
 With the definition of $B$, a further application of the Gau{\ss}-Ostrogradski theorem and the symmetry of $\nabla_{Y}\nabla_{Y}C_{\alpha}(x,Ju)$ we can reformulate this equation as 
 \begin{eqnarray*}
 \langle v,w\rangle_{\mathcal{X}\times \mathcal{X}'} &=&\langle \rho\ddot{v}-\nabla\cdot[\nabla_{Y}\nabla_{Y}C_{\alpha}(x,Ju):Jv],p\rangle_{L^{2}(0,T;H)}\\[1ex]
 &=& \langle v,\rho\ddot{p}-\nabla\cdot[\nabla_{Y}\nabla_{Y}C_{\alpha}(x,Ju):Jp]\rangle_{L^{2}(0,T;H)}\\[1ex]
 && + \Big[\int\limits_{\Omega}\langle v(t,x),\rho\dot{p}(t,x)\rangle dx\Big]_{t=0}^{T} - \Big[\int\limits_{\Omega}\langle \dot{v}(t,x),\rho p(t,x)\rangle dx\Big]_{t=0}^{T}\\
 && - \int\limits_{0}^{T}\int\limits_{\partial\Omega} \Big( p(t,\xi )\cdot \big( \nabla_{Y}\nabla_{Y}C_{\alpha} ( \xi,Ju ( t,\xi ) ):Jv(t,\xi ) \big) \Big) \cdot \nu(\xi)\,d\xi\, dt\\
 && + \int\limits_{0}^{T}\int\limits_{\partial\Omega} \Big( v(t,\xi )\cdot \big( \nabla_{Y}\nabla_{Y}C_{\alpha} ( \xi,Ju ( t,\xi ) ):Jp ( t,\xi ) \big) \Big) \cdot\nu(\xi)\,d\xi\, dt\\
 \end{eqnarray*}
 and with (\ref{vAnfangswert}) and (\ref{vRandwert}) for $v$ it follows
 \begin{eqnarray*}
 \langle v,w\rangle_{\mathcal{X}\times \mathcal{X}'} &=&\langle v,\rho\ddot{p}-\nabla\cdot[\nabla_{Y}\nabla_{Y}C_{\alpha}(x,Ju):Jp]\rangle_{L^{2}(0,T;H)}\\[1ex]
 &&+ \int\limits_{\Omega}\langle v(T,x),\rho\dot{p}(T,x)\rangle dx - \int\limits_{\Omega}\langle \dot{v}(T,x),\rho p(T,x)\rangle dx\\
 &&- \int\limits_{0}^{T}\int\limits_{\partial\Omega} \Big(\big(p(t,\xi)\cdot \nabla_{Y}\nabla_{Y}C_{\alpha}(\xi,Ju(t,\xi))\big)\cdot \nu (\xi)\Big) : Jv(t,\xi)\,d\xi\, dt.\\
 \end{eqnarray*}
 Assuming
 \begin{equation}
 p(T,x)=\dot{p}(T,x)=0\qquad x\in\Omega
 \end{equation}
 and boundary conditions
 \begin{equation}
 \big(p(t,\xi)\cdot \nabla_{Y}\nabla_{Y}C_{\alpha}(\xi,Ju(t,\xi))\big)\cdot \nu (\xi)=0\qquad (t,\xi)\in[0,T]\times\partial\Omega ,
 \end{equation}
 we finally obtain
 \[
 \langle v,w\rangle_{\mathcal{X}\times \mathcal{X}'}
 =\langle v,\rho\ddot{p}-\nabla\cdot[\nabla_{Y}\nabla_{Y}C_{\alpha}(x,Ju):Jp]\rangle_{L^{2}(0,T;H)}.
\]
Thus, $p=(B^{-1})^{*}w$ is the weak solution of the hyperbolic, backward IBVP (\ref{backward-1})--(\ref{backward-3}). We get the final representation of the adjoint $\mathcal{T}' (\alpha)^*$ as
\begin{equation*}
  \mathcal{T}'(\alpha)^{*}w = \Big[-\int\limits_{0}^{T}\int\limits_{\Omega}\langle\langle\nabla_{Y}C_{K}(x,Ju(t,x)),J(B^{-1})^{*}w(t,x)\rangle\rangle \,dx\,dt\Big]_{K=1,...,N}.
\end{equation*}
\end{proof}

\section{Conclusions}

In this article we have proven the Fr\'{e}chet differentiability of an operator, which maps the stored energy function $C=C(x,Y)$ to
the solution of a nonlinear initial-boundary value problem, that describes the dynamic behavior of hyperelastic
materials. The Fr\'{e}chet derivative has been characterized as unique (weak) solution of a linear, hyperbolic IBVP,
where we assumed, that $C$ can be written as conic combination with respect to given functions $C_K=C_K (x,Y)$
of a dictionary. An advantage of this approach is, that the $C_K$ can be determined as physically meaningful candidates for $C$, e.g. by
functions, which are polyconvex with respect to $Y$. We furthermore deduced a representation of the adjoint operator of the Fr\'{e}chet derivative which
is characterized by a linear, hyperbolic, backward IBVB. Both representations, the Fr\'{e}chet derivative of the parameter-to-solution map as well as 
its adjoint can be used for the solution of inverse problems connected to dynamic hyperelasticity, since it represents a linearization of the nonlinear
and ill-posed problem of identifying $C$ from measurements of the displacement field $u$ or part of it. Moreover the results of the article are important
for the implementation of any numerical solver of this or related problems, since those often use the concept of the Fr\'{e}chet derivative of the forward operator.
In that sense the article might have an impact to build Structural Health Monitoring systems for hyperelastic materials.

%%%%%%%%%%%%%%%%%%%%%%%%%%%%%%%%%%%%%%%%%%%

\begin{appendix}
\section{Proof of Theorem \ref{thm-glm-konv}}

\begin{proof}(of Theorem \ref{thm-glm-konv})
For proving assertion (\ref{glm-konv-ungl}) we have to adapt Theorem \ref{theorem21imp} to our problem. 
Let $\alpha\in \mathrm{int} \Drm (\mathcal{T})$. Then a (unique) solution $u(\alpha+h)$ of the IBVP (\ref{cauchy-hyper-comb}), (\ref{anfangswert1})--(\ref{randwert}) exists for $\|h\|_\infty$ sufficiently small. Obviously $\tilde{u}:=u(\alpha+h)-u(\alpha)$ solves equation
\begin{eqnarray}\label{utilde}
 &&\rho\ddot{\tilde{u}}(t,x) - \sum_{K=1}^{N}\alpha_{K}\nabla\cdot[\nabla_{Y}C_{K}(x,Ju(\alpha+h))-\nabla_{Y}C_{K}(x,Ju(\alpha))] \nonumber \\ 
 &&= \nabla\cdot [\nabla_{Y}C_{h}(x,Ju(\alpha+h))] 
\end{eqnarray}
for $t\in[0,T]$ and $x\in\Omega$ with initial conditions
\begin{equation}\label{utildeAnfangswert}
  \tilde{u}(0,x)=\dot{\tilde{u}}(0,x)=0\qquad \mbox{for } x\in\Omega
\end{equation}
and vanishing boundary values 
\begin{equation}\label{utildeRandwert}
 \tilde{u}(t,x)=0\qquad \mbox{for all } x\in\partial\Omega\,.
\end{equation}
Theorem \ref{theorem21imp} tells then that
\begin{eqnarray}\label{theorem21spez}
 &&\bigg[\rho\|\dot{\tilde{u}}(t,\cdot)\|_{L^{2}(\Omega,\RR^3)}^{2} + \kappa(\alpha)\|J\tilde{u}(t,\cdot)\|_{L^{2}(\Omega,\RR^{3\times 3})}^{2} + \rho\|\ddot{\tilde{u}}(t,\cdot)\|_{L^{2}(\Omega,\RR^3)}^{2} \nonumber \\
 &&+ \kappa(\alpha)\|J\dot{\tilde{u}}(t,\cdot)\|_{L^{2}(\Omega,\RR^{3\times 3})}^{2} + \|\tilde{u}(t,\cdot)\|_{H^{2}(\Omega,\RR^3)}^{2}\bigg]^{\frac{1}{2}} \nonumber \\
 &&\leq \bar{C}_{2}\|h\|_{\infty}.
\end{eqnarray}
Define $d:=u(\alpha+h)-u(\alpha)-v$. Then $d$ solves due to (\ref{gateaux-diff}) and (\ref{utilde})
  \begin{eqnarray*}\scriptscriptstyle
  &&\rho\ddot{d}(t,\cdot) - \sum_{K=1}^{N}\alpha_{K}\nabla\cdot[\nabla_{Y}\nabla_{Y}C_{K}(x,Ju(\alpha)):Jd(t,\cdot)]\\
  &&= \sum_{K=1}^{N}\alpha_{K}\nabla\cdot\Big[\nabla_{Y}C_{K}(x,Ju(\alpha+h))-\nabla_{Y}C_{K}(x,Ju(\alpha))\\
  &&\hspace*{2cm}-\nabla_{Y}\nabla_{Y}C_{K}(x,Ju(\alpha)):J\tilde{u}(t,\cdot)\Big]\\
  &&\hspace*{1cm}+ \sum_{K=1}^{N}h_{K}\nabla\cdot[\nabla_{Y}C_{K}(x,Ju(\alpha+h))-\nabla_{Y}C_{K}(x,Ju(\alpha))].
 \end{eqnarray*}
 Multiplying this by $2\dot{d}$ and integrating over $\Omega$ yields by means of the Gau{\ss}-Ostrogradski theorem with (\ref{vRandwert}) as well as (\ref{utildeRandwert}) and 
 $Y_{p}:=pJu(\alpha+h)+(1-p)Ju(\alpha)$ for $p\in[0,1]$
 \begin{eqnarray}\label{A-and-B-term}
 &&2\langle\rho\ddot{d}(t,\cdot),\dot{d}(t,\cdot)\rangle_{L^{2}(\Omega,\RR^3)}
 +2\underbrace{\sum_{K=1}^{N}\alpha_{K}\int\limits_{\Omega}\langle\langle\nabla_{Y}\nabla_{Y}C_{K}(x,Ju(\alpha)):Jd(t,\cdot),J\dot{d}(t,\cdot)\rangle\rangle dx}_{a_{1}(t;d,\dot{d})}\nonumber \\
 &&= -2\underbrace{\sum_{K=1}^{N}\alpha_{K}\int\limits_{\Omega}\int\limits_{0}^{1}\int\limits_{0}^{1}s\langle\langle[\nabla_{Y}\nabla_{Y}C_{K}(x,Y_{rs}):J\tilde{u}(t,\cdot)]:J\tilde{u}(t,\cdot),J\dot{d}(t,\cdot)\rangle\rangle drdsdx}_{(A)}\nonumber \\
 &&\;\;\; -2\underbrace{\sum_{K=1}^{N}h_{K}\int\limits_{\Omega}\int\limits_{0}^{1}\langle\langle\nabla_{Y}\nabla_{Y}C_{K}(x,Y_{s}):J\tilde{u}(t,\cdot),J\dot{d}(t,\cdot)\rangle\rangle dsdx}_{(B)}.
 \end{eqnarray}
 Compared to (\ref{vschwach}) the differential equation for $d$ differs only with respect to the right side. For this reason it follows (cf. proof of Theorem \ref{gateaux-stetig}) for $\tau\in[0,T]$ 
 \begin{equation}\label{gleichung-d}
 a_{1}(t;d,d) + \rho\|\dot{d}(\tau,\cdot)\|_{L^{2}(\Omega,\RR^3)}^{2} = \int\limits_{0}^{\tau}a'_{1}(t;d,d)dt + 2\int\limits_{0}^{\tau}[-(A)-(B)]dt.
\end{equation}
Inserting (\ref{a1}) and (\ref{a1'}) into this equation yields
\begin{eqnarray}\label{ungleichung-d}
 &&\sum_{K=1}^{N}\alpha_{K}\kappa_{K}^{[1]}\|Jd(\tau,\cdot)\|_{L^{2}(\Omega,\RR^{3\times3})}^{2} + \rho\|\dot{d}(\tau,\cdot)\|_{L^{2}(\Omega,\RR^3)}^{2} \nonumber \\
 &&\leq  \int\limits_{0}^{\tau}\frac{729}{8}\eta\mu(\alpha)M_{1}\|Jd(t,\cdot)\|_{L^{2}(\Omega,\RR^{3\times3})}^{2}dt + 2\int\limits_{0}^{\tau}[-(A)-(B)]dt.
\end{eqnarray}
In the following we drop the argument $(t,\cdot)$ for the sake of a better readability and use $\alpha_{K}> 0$ for all $K=1,...,N$.\\
Applying (\ref{beschr1}), (\ref{beschr2}) and (\ref{bedA}) we deduce
\begin{eqnarray*}\scriptscriptstyle
 &&\Big|\int\limits_{0}^{\tau}(A)dt\Big|\\
 &&= \Big|\sum_{K=1}^{N}\alpha_{K}\int\limits_{0}^{\tau}\int\limits_{\Omega}\int\limits_{0}^{1}\int\limits_{0}^{1}s\langle\langle\nabla_{Y}\nabla_{Y}C_{K}(x,Y_{rs}):J\tilde{u}:J\tilde{u},J\dot{d}\rangle\rangle \, dr\,ds\,dx\,dt\Big|\\
 %&&= |\sum_{K=1}^{N}\alpha_{K}\int\limits_{0}^{\tau}\int\limits_{\Omega}\int\limits_{0}^{1}\int\limits_{0}^{1}s\sum_{\substack{p,q,\\k,l,\\i,j=1}}^{3}\partial_{Y_{pq}}\partial_{Y_{ij}}\partial_{Y_{kl}}C_{K}(x,Y_{rs})\partial_{q}\tilde{u}_{p}\partial_{l}\tilde{u}_{k}\partial_{j}\dot{d}_{i}drdsdxdt|\\
 &&\leq \Big|\sum_{K=1}^{N}\alpha_{K}\int\limits_{0}^{\tau}\partial_{t} \Big[\int\limits_{\Omega}\int\limits_{0}^{1}\int\limits_{0}^{1}s\sum_{\substack{p,q,\\k,l,\\i,j=1}}^{3}\partial_{Y_{pq}}\partial_{Y_{ij}}\partial_{Y_{kl}}C_{K}(x,Y_{rs})\partial_{q}\tilde{u}_{p}\partial_{l}\tilde{u}_{k}\partial_{j}d_{i}\, dr\, ds\, dx \Big]\, dt\Big|\\
 &&\hspace*{1cm}+\sum_{K=1}^{N}\alpha_{K}\int\limits_{0}^{\tau}\int\limits_{\Omega}\int\limits_{0}^{1}\int\limits_{0}^{1}s\sum_{\substack{a,b,\\p,q,\\k,l,\\i,j=1}}^{3}|\partial_{Y_{ab}}\partial_{Y_{pq}}\partial_{Y_{ij}}\partial_{Y_{kl}}C_{K}(x,Y_{rs})|\times\\
 &&\hspace*{2cm}\times|rs\partial_{b}\dot{u}_{a}(\alpha+h)+(1-rs)\partial_{b}\dot{u}_{a}(\alpha)||\partial_{q}\tilde{u}_{p}||\partial_{l}\tilde{u}_{k}||\partial_{j}d_{i}|\, dr\, ds\, dx\, dt\\
 &&\hspace*{1cm}+\sum_{K=1}^{N}\alpha_{K}\int\limits_{0}^{\tau}\int\limits_{\Omega}\int\limits_{0}^{1}\int\limits_{0}^{1}s\sum_{\substack{p,q,\\k,l,\\i,j=1}}^{3}|\partial_{Y_{pq}}\partial_{Y_{ij}}\partial_{Y_{kl}}C_{K}(x,Y_{rs})||\partial_{q}\dot{\tilde{u}}_{p}||\partial_{l}\tilde{u}_{k}||\partial_{j}d_{i}|\, dr\, ds\, dx\, dt\\
 &&\hspace*{1cm}+\sum_{K=1}^{N}\alpha_{K}\int\limits_{0}^{\tau}\int\limits_{\Omega}\int\limits_{0}^{1}\int\limits_{0}^{1}s\sum_{\substack{p,q,\\k,l,\\i,j=1}}^{3}|\partial_{Y_{pq}}\partial_{Y_{ij}}\partial_{Y_{kl}}C_{K}(x,Y_{rs})||\partial_{q}\tilde{u}_{p}||\partial_{l}\dot{\tilde{u}}_{k}||\partial_{j}d_{i}|\, dr\, ds\, dx\, dt.\\
 \end{eqnarray*}
Taking the initial conditions for $\tilde{u}$ into account gives 
 \begin{eqnarray*}\scriptscriptstyle
 &&\Big|\int\limits_{0}^{\tau}(A)dt\Big|\\
  &&\leq \sum_{K=1}^{N}\alpha_{K}\int\limits_{\Omega}\int\limits_{0}^{1}\int\limits_{0}^{1}s\sum_{\substack{p,q,\\k,l,\\i,j=1}}^{3}|\partial_{Y_{pq}}\partial_{Y_{ij}}\partial_{Y_{kl}}C_{K}(x,Y_{rs})||\partial_{q}\tilde{u}_{p}(\tau)||\partial_{l}\tilde{u}_{k}(\tau)||\partial_{j}d_{i}(\tau)|\, dr\, ds\, dx\\
  &&\hspace*{1cm}+\sum_{K=1}^{N}\alpha_{K}\mu_{K}^{[3]}M_{1}\int\limits_{0}^{\tau}\int\limits_{\Omega}\int\limits_{0}^{1}\int\limits_{0}^{1}s\sum_{\substack{a,b,\\p,q,\\k,l,\\i,j=1}}^{3}|\partial_{q}\tilde{u}_{p}||\partial_{l}\tilde{u}_{k}||\partial_{j}d_{i}|\, dr\, ds\, dx\, dt\\
  &&\hspace*{1cm}+2\sum_{K=1}^{N}\alpha_{K}\mu_{K}^{[2]}\int\limits_{0}^{\tau}\int\limits_{\Omega}\int\limits_{0}^{1}\int\limits_{0}^{1}s\sum_{\substack{p,q,\\k,l,\\i,j=1}}^{3}|\partial_{q}\tilde{u}_{p}||\partial_{l}\dot{\tilde{u}}_{k}||\partial_{j}d_{i}|\, dr\, ds\, dx\, dt\\
  &&\leq \frac{1}{2} \sum_{K=1}^{N}\alpha_{K}\mu_{K}^{[2]}\int\limits_{\Omega}\bigg(\sum_{k,l=1}^{3}|\partial_{l}\tilde{u}_{k}(\tau)|\bigg)^{2}\sum_{i,j=1}^{3}|\partial_{j}d_{i}(\tau)|\, dx\\
  &&\hspace*{1cm}+\frac{9}{2}M_{1}\sum_{K=1}^{N}\alpha_{K}\mu_{K}^{[3]}\int\limits_{0}^{\tau}\int\limits_{\Omega}\bigg(\sum_{k,l=1}^{3}|\partial_{l}\tilde{u}_{k}|\bigg)^{2}\sum_{i,j=1}^{3}|\partial_{j}d_{i}|\, dx\, dt\\
  &&\hspace*{1cm}+\sum_{K=1}^{N}\alpha_{K}\mu_{K}^{[2]}\int\limits_{0}^{\tau}\int\limits_{\Omega}\bigg(\sum_{p,q=1}^{3}|\partial_{q}\tilde{u}_{p}|\bigg)\bigg(\sum_{k,l=1}^{3}|\partial_{l}\dot{\tilde{u}}_{k}|\bigg)\bigg(\sum_{i,j=1}^{3}|\partial_{j}d_{i}|\bigg)dx\, dt.\\
 \end{eqnarray*}
 Multiple applications of H\"{o}lder's inequality yield
 \begin{eqnarray*}\scriptscriptstyle
 &&\Big|\int\limits_{0}^{\tau}(A)dt\Big|\\
  &&\leq \frac{1}{2} \sum_{K=1}^{N}\alpha_{K}\mu_{K}^{[2]}\bigg(\int\limits_{\Omega}\Big(\sum_{k,l=1}^{3}|\partial_{l}\tilde{u}_{k}(\tau)|\Big)^{4}dx\bigg)^{\frac{1}{2}}\bigg(\int\limits_{\Omega}\Big(\sum_{i,j=1}^{3}|\partial_{j}d_{i}(\tau)|\Big)^{2}dx\bigg)^{\frac{1}{2}}\\
  &&\hspace*{1cm}+\frac{9}{2}M_{1}\sum_{K=1}^{N}\alpha_{K}\mu_{K}^{[3]}\int\limits_{0}^{\tau}\bigg(\int\limits_{\Omega}\Big(\sum_{k,l=1}^{3}|\partial_{l}\tilde{u}_{k}|\Big)^{4}dx\bigg)^{\frac{1}{2}}\bigg(\int\limits_{\Omega}\Big(\sum_{i,j=1}^{3}|\partial_{j}d_{i}|\Big)^{2}dx\bigg)^{\frac{1}{2}}dt\\
  &&+\sum_{K=1}^{N}\alpha_{K}\mu_{K}^{[2]}\int\limits_{0}^{\tau}\bigg(\int\limits_{\Omega}\Big(\sum_{p,q=1}^{3}|\partial_{q}\tilde{u}_{p}|\Big)^{4}dx\bigg)^{\frac{1}{4}}\bigg(\int\limits_{\Omega}\Big(\sum_{k,l=1}^{3}|\partial_{l}\dot{\tilde{u}}_{k}|\Big)^{4}dx\bigg)^{\frac{1}{4}}
  \bigg(\int\limits_{\Omega}\Big(\sum_{i,j=1}^{3}|\partial_{j}d_{i}|\Big)^{2}dx\bigg)^{\frac{1}{2}}dt\\
  &&\leq \frac{1}{2} \sum_{K=1}^{N}\alpha_{K}\mu_{K}^{[2]}\bigg(\int\limits_{\Omega}\Big(\sum_{k,l=1}^{3}|\partial_{l}\tilde{u}_{k}(\tau)|\Big)^{4}dx\bigg)^{\frac{1}{2}}\bigg(9\int\limits_{\Omega}\sum_{i,j=1}^{3}|\partial_{j}d_{i}(\tau)|^{2}dx\bigg)^{\frac{1}{2}}\\
  &&\hspace*{1cm}+\frac{9}{2}M_{1}\sum_{K=1}^{N}\alpha_{K}\mu_{K}^{[3]}\int\limits_{0}^{\tau}\bigg(\int\limits_{\Omega}\Big(\sum_{k,l=1}^{3}|\partial_{l}\tilde{u}_{k}|\Big)^{4}dx\bigg)^{\frac{1}{2}}\bigg(9\int\limits_{\Omega}\sum_{i,j=1}^{3}|\partial_{j}d_{i}|^{2}dx\bigg)^{\frac{1}{2}}dt\\
  &&\hspace*{5mm}+\sum_{K=1}^{N}\alpha_{K}\mu_{K}^{[2]}\int\limits_{0}^{\tau}\bigg(\int\limits_{\Omega}\Big(\sum_{p,q=1}^{3}|\partial_{q}\tilde{u}_{p}|\Big)^{4}dx\bigg)^{\frac{1}{4}}\bigg(9\int\limits_{\Omega}\sum_{k,l=1}^{3}|\partial_{l}\dot{\tilde{u}}_{k}|^{4}dx\bigg)^{\frac{1}{4}}
  \bigg(9\int\limits_{\Omega}\sum_{i,j=1}^{3}|\partial_{j}d_{i}|^{2}dx\bigg)^{\frac{1}{2}}dt.\\
 \end{eqnarray*}
 From (\ref{ungl-h2}) we further deduce
 \begin{eqnarray*}\scriptscriptstyle
  &&\Big|\int\limits_{0}^{\tau}(A)dt\Big|\\
  &&\leq \frac{81}{2}(1+C_{\Omega})C_{SE}^{2}\sum_{K=1}^{N}\alpha_{K}\mu_{K}^{[2]}\|\tilde{u}(\tau,\cdot)\|_{H^{2}(\Omega,\RR^3)}^{2} \|d(\tau,\cdot)\|_{H_{0}^{1}(\Omega,\RR^3)}\\
  &&\hspace*{5mm}+\frac{729}{2}(1+C_{\Omega})C_{SE}^{2}M_{1}\sum_{K=1}^{N}\alpha_{K}\mu_{K}^{[3]}\int\limits_{0}^{\tau}\|\tilde{u}(t,\cdot)\|_{H^{2}(\Omega,\RR^3)}^{2}\|d(t,\cdot)\|_{H_{0}^{1}(\Omega,\RR^3)}dt\\
  &&\hspace*{5mm}+81\sqrt{1+C_{\Omega}}C_{SE}\sum_{K=1}^{N}\alpha_{K}\mu_{K}^{[2]}\int\limits_{0}^{\tau}\|\tilde{u}(t,\cdot)\|_{H^{2}(\Omega,\RR^3)}\|d(t,\cdot)\|_{H_{0}^{1}(\Omega,\RR^3)}
 \bigg(\int\limits_{\Omega}\sum_{k,l=1}^{3}|\partial_{l}\dot{\tilde{u}}_{k}|^{4}dx\bigg)^{\frac{1}{4}}dt,\\
 \end{eqnarray*}
 where we used the estimate
 \begin{equation}\label{d-estimate} \|d(\tau,\cdot)\|_{H_{0}^{1}(\Omega,\RR^3)}  \leq \|\tilde{u}(\tau,\cdot)\|_{H^{2}(\Omega,\RR^3)}+\|v(\tau,\cdot)\|_{H_{0}^{1}(\Omega,\RR^3)}. \end{equation}
 Using (\ref{ungl4-2}) we can further estimate this by
 \begin{eqnarray*}\scriptscriptstyle
  &&\Big|\int\limits_{0}^{\tau}(A)dt\Big|\\
  &&\leq \frac{81}{2}(1+C_{\Omega})C_{SE}^{2}\sum_{K=1}^{N}\alpha_{K}\mu_{K}^{[2]}\|\tilde{u}(\tau,\cdot)\|_{H^{2}(\Omega,\RR^3)}^{2}\bigg(\|\tilde{u}(\tau,\cdot)\|_{H^{2}(\Omega,\RR^3)}+\|v(\tau,\cdot)\|_{U}\bigg)\\
  &&+\frac{729}{2}(1+C_{\Omega})C_{SE}^{2}M_{1}\sum_{K=1}^{N}\alpha_{K}\mu_{K}^{[3]}\int\limits_{0}^{\tau}\|\tilde{u}(t,\cdot)\|_{H^{2}(\Omega,\RR^3)}^{2}\|d(t,\cdot)\|_{H_{0}^{1}(\Omega,\RR^3)}dt\\
  &&+81\sqrt{2M_{1}(1+C_{\Omega})}C_{SE}\sum_{K=1}^{N}\alpha_{K}\mu_{K}^{[2]}\int\limits_{0}^{\tau}\|\tilde{u}(t,\cdot)\|_{H^{2}(\Omega,\RR^3)}\|d(t,\cdot)\|_{H_{0}^{1}(\Omega,\RR^3)}
  \|J\dot{\tilde{u}}(t,\cdot)\|_{L^{2}(\Omega,\RR^{3\times 3})}^{\frac{1}{2}}dt\\
 \end{eqnarray*}
 and with (\ref{vUungl}) as well as (\ref{theorem21spez}) we finally obtain
 \begin{eqnarray*}\scriptscriptstyle
 &&\Big|\int\limits_{0}^{\tau}(A)dt\Big|\\
  &&\leq \frac{81}{2}(1+C_{\Omega})C_{SE}^{2}\sum_{K=1}^{N}\alpha_{K}\mu_{K}^{[2]}\bar{C}_{2}^2 \|h\|_{\infty}^{2}(\bar{C}_{2}+\bar{L}_{1}S(T))\|h\|_{\infty}\\
  &&\hspace*{3mm}+\frac{729}{2}(1+C_{\Omega})C_{SE}^{2}M_{1}\bar{C}_{2}^{2}\sum_{K=1}^{N}\alpha_{K}\mu_{K}^{[3]}\|h\|_{\infty}^{2}\int\limits_{0}^{\tau}\|d(t,\cdot)\|_{H_{0}^{1}(\Omega,\RR^3)}dt\\
  &&\hspace*{3mm}+81\sqrt{\frac{2M_{1}(1+C_{\Omega})}{\kappa(\alpha)^\frac{1}{2}}}C_{SE}\bar{C}_{2}^{\frac{3}{2}}\sum_{K=1}^{N}\alpha_{K}\mu_{K}^{[2]}\|h\|_{\infty}^{\frac{3}{2}}\int\limits_{0}^{\tau}\|d(t,\cdot)\|_{H_{0}^{1}(\Omega,\RR^3)}dt\\
  &&=\tilde{B}_{1}\|h\|_{\infty}^{3}+\tilde{B}_{2}(h)\|h\|_{\infty}^{\frac{3}{2}}\int\limits_{0}^{\tau}\|d(t,\cdot)\|_{H_{0}^{1}(\Omega,\RR^3)}dt,\\
 \end{eqnarray*}
  where we set
  \begin{equation*}
   \tilde{B}_{1}:=\frac{81}{2}(1+C_{\Omega})C_{SE}^{2}\bar{C}_{2}^2 (\bar{C}_{2}+\bar{L}_{1}S(T))\sum_{K=1}^{N}\alpha_{K}\mu_{K}^{[2]}
  \end{equation*}
  and
  \begin{eqnarray*}
  \tilde{B}_{2}(h) &:=&  81\sqrt{\frac{2M_{1}(1+C_{\Omega})}{\kappa(\alpha)^\frac{1}{2}}}C_{SE}\bar{C}_{2}^{\frac{3}{2}}\sum_{K=1}^{N}\alpha_{K}\mu_{K}^{[2]}\\
  && +\frac{729}{2}(1+C_{\Omega})C_{SE}^{2}M_{1}\bar{C}_{2}^{2}\sum_{K=1}^{N}\alpha_{K}\mu_{K}^{[3]}\|h\|_{\infty}^{\frac{1}{2}}. 
  \end{eqnarray*}
  Here, $S(T)$ is given by (\ref{S-tau}).
  Recalling the results from Section \ref{secVor} and Theorem \ref{theorem21imp} we may estimate
  \begin{equation}\label{B1}
   \tilde{B}_{1}\leq B_{1}:=\frac{729}{16}(1+C_{\Omega})C_{SE}^{2}\bar{C}_{2}^2 (\bar{C}_{2}+\bar{L}_{1}S(T))\eta\mu(\alpha)
  \end{equation}
  respectively
  \begin{eqnarray}\label{B2}
  \tilde{B}_{2}(h)\leq B_{2}(h)&:=&\frac{729}{8}\sqrt{\frac{2M_{1}(1+C_{\Omega})}{\kappa(\alpha)^\frac{1}{2}}}C_{SE}\bar{C}_{2}^{\frac{3}{2}}\eta\mu(\alpha) \nonumber \\
  &&+\frac{729}{2}(1+C_{\Omega})C_{SE}^{2}M_{1}\bar{C}_{2}^{2}\sum_{K=1}^{N}\alpha_{K}\mu_{K}^{[3]}\|h\|_{\infty}^{\frac{1}{2}}. 
  \end{eqnarray}
  After all this we arrived at the estimate
  \begin{equation}\label{int-(A)}
   \Big|\int\limits_{0}^{\tau}(A)dt\Big| \leq B_{1}\|h\|_{\infty}^{3}+B_{2}(h)\|h\|_{\infty}^{\frac{3}{2}}\int\limits_{0}^{\tau}\|d(t,\cdot)\|_{H_{0}^{1}(\Omega,\RR^3)}dt
  \end{equation}
  for the $(A)$-term. Please note that for fixed $\alpha$ we have the convergence
  \begin{equation}\label{B-h-zero}
  \lim_{\|h\|_\infty\to 0} B_2 (h) = \frac{729}{8}\sqrt{\frac{2M_{1}(1+C_{\Omega})}{\kappa(\alpha)^\frac{1}{2}}}C_{SE}\bar{C}_{2}^{\frac{3}{2}}\eta\mu(\alpha)=: \bar{B}_2
  \end{equation}
  with a constant $\bar{B}_2$ independent of $h$.\\
  Our investigations are now devoted to the $(B)$-term from (\ref{A-and-B-term}). We compute
  \begin{eqnarray*}
  \Big|\int\limits_{0}^{\tau}(B)dt \Big|
  &=&\Big|\int\limits_{0}^{\tau}\sum_{K=1}^{N}h_{K}\int\limits_{\Omega}\int\limits_{0}^{1}\langle\langle\nabla_{Y}\nabla_{Y}C_{K}(x,Y_{s}):J\tilde{u},J\dot{d}\rangle\rangle ds\, dx\, dt \Big|\\
  %&&=|\int\limits_{0}^{\tau}\partial_{t}[\sum_{K=1}^{N}h_{K}\int\limits_{\Omega}\int\limits_{0}^{1}\langle\langle\nabla_{Y}\nabla_{Y}C_{K}(x,Y_{s}):J\tilde{u},Jd\rangle\rangle dsdx]dt\\
  %&&\;\;-\int\limits_{0}^{\tau}\sum_{K=1}^{N}h_{K}\int\limits_{\Omega}\int\limits_{0}^{1}\langle\langle\nabla_{Y}\nabla_{Y}\nabla_{Y}C_{K}(x,Y_{s}):\partial_{t}Y_{s}:J\tilde{u},Jd\rangle\rangle dsdxdt \\
  %&&\;\;-\int\limits_{0}^{\tau}\sum_{K=1}^{N}h_{K}\int\limits_{\Omega}\int\limits_{0}^{1}\langle\langle\nabla_{Y}\nabla_{Y}C_{K}(x,Y_{s}):J\dot{\tilde{u}},Jd\rangle\rangle dsdxdt|\\
  &\leq & \Big|\int\limits_{0}^{\tau}\partial_{t}\Big[\sum_{K=1}^{N}h_{K}\sum_{i,j,k,l=1}^{N}\int\limits_{\Omega}\int\limits_{0}^{1}\partial_{Y_{ij}}\partial_{Y_{kl}}C_{K}(x,Y_{s})\partial_{j}\tilde{u}_{i}\partial_{l}d_{k} ds\, dx \Big]\,dt \Big|\\
  &&\hspace*{1cm}+\sum_{\substack{i,j,k,\\l,p,q=1}}^{3}\int\limits_{0}^{\tau}\int\limits_{\Omega}\int\limits_{0}^{1}\sum_{K=1}^{N}|h_K||\partial_{Y_{pq}}\partial_{Y_{ij}}\partial_{Y_{kl}}C_{K}(x,Y_{s})|\times\\[1ex]
  &&\hspace*{1.5cm}\times \big(s |\partial_{q}\dot{u}_{p}(\alpha+h)|+(1-s)|\partial_{q}\dot{u}_{p}(\alpha)| \big)|\partial_{j}\tilde{u}_{i}||\partial_{l}d_{k}|\, ds\, dx\, dt\\
  &&\hspace*{1cm}+\sum_{i,j,k,l=1}^{3}\int\limits_{0}^{\tau}\int\limits_{\Omega}\int\limits_{0}^{1}\sum_{K=1}^{N}|h_K||\partial_{Y_{ij}}\partial_{Y_{kl}}C_{K}(x,Y_{s})||\partial_{j}\dot{\tilde{u}}_{i}||\partial_{l}d_{k}|\, ds\, dx\, dt.\\
  \end{eqnarray*}
 Using (\ref{beschr2}), (\ref{bed2}), (\ref{bedA}), H\"{o}lder's inequality and the initial condition for $\tilde{u}$ we can this further estimate as
  \begin{eqnarray*}
  \Big|\int\limits_{0}^{\tau}(B)dt \Big|
  &\leq &\sum_{K=1}^{N}|h_K|\mu_{K}^{[1]}\int\limits_{\Omega}\sum_{i,j,k,l=1}^{3}|\partial_{j}\tilde{u}_{i}(\tau)||\partial_{l}d_{k}(\tau)|\,dx\\
  &&\hspace*{1cm}+ 9M_{1}\sum_{K=1}^{N}|h_K|\mu_{K}^{[2]}\int\limits_{0}^{\tau}\int\limits_{\Omega}\sum_{i,j,k,l=1}^{3}|\partial_{j}\tilde{u}_{i}||\partial_{l}d_{k}|\,dx\,dt\\
  &&\hspace*{1cm}+ \sum_{K=1}^{N}|h_K|\mu_{K}^{[1]}\int\limits_{0}^{\tau}\int\limits_{\Omega}\sum_{i,j,k,l=1}^{3}|\partial_{j}\dot{\tilde{u}}_{i}||\partial_{l}d_{k}|\,dx\,dt\\
  &&\leq 9\sum_{K=1}^{N}|h_K|\mu_{K}^{[1]}\|J\tilde{u}(\tau,\cdot)\|_{L^{2}(\Omega,\RR^{3\times 3})}\|d(\tau,\cdot)\|_{H_{0}^{1}(\Omega,\RR^3)}\\
  &&\hspace*{1cm}+ 81M_{1}\sum_{K=1}^{N}|h_K|\mu_{K}^{[2]}\int\limits_{0}^{\tau}\|J\tilde{u}(t,\cdot)\|_{L^{2}(\Omega,\RR^{3\times 3})}\|d(t,\cdot)\|_{H_{0}^{1}(\Omega,\RR^3)}dt\\
  &&\hspace*{1cm}+ 9\sum_{K=1}^{N}|h_K|\mu_{K}^{[1]}\int\limits_{0}^{\tau}\|J\dot{\tilde{u}}(t,\cdot)\|_{L^{2}(\Omega,\RR^{3\times 3})}\|d(t,\cdot)\|_{H_{0}^{1}(\Omega,\RR^3)}dt.\\
  \end{eqnarray*}
 An application of (\ref{theorem21spez}) along with (\ref{d-estimate}) furthermore yields
 \begin{eqnarray*}
  %&&\leq 9\sum_{K=1}^{N}\mu_{K}^{[1]}\frac{\bar{C}_{2}}{\sqrt{\kappa(\alpha)}}(\bar{C}_{2}+\bar{L}_{1}S(T))\|h\|_{\infty}^{3}\\
  %&&+ 81M_{1}\sum_{K=1}^{N}\mu_{K}^{[2]}\frac{\bar{C}_{2}}{\sqrt{\kappa(\alpha)}}\|h\|_{\infty}^{2}\int\limits_{0}^{\tau}\|d(t,\cdot)\|_{H_{0}^{1}(\Omega,\RR^3)}dt\\
  %&&+ 9\sum_{K=1}^{N}\mu_{K}^{[1]}\frac{\bar{C}_{2}}{\sqrt{\kappa(\alpha)}}\|h\|_{\infty}^{2}\int\limits_{0}^{\tau}\|d(t,\cdot)\|_{H_{0}^{1}(\Omega,\RR^3)}dt\\
  \Big|\int\limits_{0}^{\tau}(B)dt \Big|
  &\leq & 9\sum_{K=1}^{N}\mu_{K}^{[1]}\frac{\bar{C}_{2}}{\sqrt{\kappa(\alpha)}}(\bar{C}_{2}+\bar{L}_{1}S(T))\|h\|_{\infty}^{3}\\
  &&+ 9\frac{\bar{C}_{2}}{\sqrt{\kappa(\alpha)}}\Big(9M_{1}\sum_{K=1}^{N}\mu_{K}^{[2]}+\sum_{K=1}^{N}\mu_{K}^{[1]}\Big)\|h\|_{\infty}^{2}\int\limits_{0}^{\tau}\|d(t,\cdot)\|_{H_{0}^{1}(\Omega,\RR^3)}dt.\\
 \end{eqnarray*}
 Setting
 \begin{equation}\label{B3}
  B_{3}:= 9\sum_{K=1}^{N}\mu_{K}^{[1]}\frac{\bar{C}_{2}}{\sqrt{\kappa(\alpha)}}(\bar{C}_{2}+\bar{L}_{1}S(T))
 \end{equation}
 and
 \begin{equation}\label{B4}
  B_{4}:= 9\frac{\bar{C}_{2}}{\sqrt{\kappa(\alpha)}}(9M_{1}\sum_{K=1}^{N}\mu_{K}^{[2]}+\sum_{K=1}^{N}\mu_{K}^{[1]}),
 \end{equation}
 we obtain
 \begin{equation}\label{int-(B)}
  \Big|\int\limits_{0}^{\tau}(B)dt \Big| \leq B_{3}\|h\|_{\infty}^{3} + B_{4}\|h\|_{\infty}^{2}\int\limits_{0}^{\tau}\|d(t,\cdot)\|_{H_{0}^{1}(\Omega,\RR^3)}dt
 \end{equation}
 as according estimate for the $(B)$-term.\\
 Summarizing all our efforts so far, we obtain from (\ref{ungleichung-d}) with (\ref{int-(A)}) and (\ref{int-(B)})
 \begin{eqnarray*}
 \kappa(\alpha)\|d(\tau,\cdot)\|_{H_{0}^{1}(\Omega,\RR^3)}^{2} +  \rho\|\dot{d}(\tau,\cdot)\|_{L^{2}(\Omega,\RR^3)}^{2}
 &\leq & \frac{729\mu(\alpha)}{8}\eta M_{1}\int\limits_{0}^{\tau}\|d(t,\cdot)\|_{H_{0}^{1}(\Omega,\RR^3)}^{2}dt\\
 &&\hspace*{-2cm}+2B_{1}\|h\|_{\infty}^{3} + 2B_{2}(h)\|h\|_{\infty}^{\frac{3}{2}}\int\limits_{0}^{\tau}\|d(t,\cdot)\|_{H_{0}^{1}(\Omega,\RR^3)}dt\\
 &&\hspace*{-2cm}+2B_{3}\|h\|_{\infty}^{3} + 2B_{4}\|h\|_{\infty}^{2}\int\limits_{0}^{\tau}\|d(t,\cdot)\|_{H_{0}^{1}(\Omega,\RR^3)}dt
 \end{eqnarray*}
 and thusly
 \begin{eqnarray*}
  && \kappa(\alpha)\|d(\tau,\cdot)\|_{H_{0}^{1}(\Omega,\RR^3)}^{2} +  \rho\|\dot{d}(\tau,\cdot)\|_{L^{2}(\Omega,\RR^3)}^{2}\\
  &&\hspace*{1cm}\leq \frac{729\mu(\alpha)}{8\kappa(\alpha)}\eta M_{1}\int\limits_{0}^{\tau}\kappa(\alpha)\|d(t,\cdot)\|_{H_{0}^{1}(\Omega,\RR^3)}^{2} +  \rho\|\dot{d}(t,\cdot)\|_{L^{2}(\Omega,\RR^3)}^{2}dt\\
  &&\hspace*{1.5cm}+ 2(B_{1}+B_{3})\|h\|_{\infty}^{3} + 2(B_{2}(h)+B_{4}\|h\|_{\infty}^{\frac{1}{2}})\|h\|_{\infty}^{\frac{3}{2}}\frac{1}{\sqrt{\kappa(\alpha)}}\times\\
  &&\hspace*{2cm}\times\int\limits_{0}^{\tau}\bigg(\kappa(\alpha)\|d(t,\cdot)\|_{H_{0}^{1}(\Omega,\RR^3)}^{2} +  \rho\|\dot{d}(t,\cdot)\|_{L^{2}(\Omega,\RR^3)}^{2}\bigg)^{\frac{1}{2}}dt.\\
 \end{eqnarray*}
 The estimate above has now a form that we can apply Gronwall's lemma in the particular situation (\ref{gronwall-const}) with the settings 
 \[\psi := \kappa(\alpha)\|d(\tau,\cdot)\|_{H_{0}^{1}(\Omega,\RR^3)}^{2} +  \rho\|\dot{d}(\tau,\cdot)\|_{L^{2}(\Omega,\RR^3)}^{2},\]
 \[a := 2(B_{1}+B_{3})\|h\|_{\infty}^{3},\]
 \[b := \frac{729\mu(\alpha)}{8\kappa(\alpha)}\eta M_{1}>0\]
 and
 \[k:= 2(B_{2}(h)+B_{4}\|h\|_{\infty}^{\frac{1}{2}})\|h\|_{\infty}^{\frac{3}{2}}\frac{1}{\sqrt{\kappa(\alpha)}} . \]
 This yields for $\tau\in[0,T]$
 \[
  \kappa(\alpha)\|d(\tau,\cdot)\|_{H_{0}^{1}(\Omega,\RR^3)}^{2} +  \|\dot{d}(\tau,\cdot)\|_{L^{2}(\Omega,\RR^3)}^{2}
  \leq \bigg[\exp(\frac{1}{2}b\tau)a^{\frac{1}{2}} + \frac{k}{b}(\exp(\frac{1}{2}b\tau)-1)\bigg]^{2}
 \]
 and using the setting for the constants above 
 \begin{eqnarray*}
  \|d(\tau,\cdot)\|_{H_{0}^{1}(\Omega,\RR^3)} &\leq &\\
  &&\hspace*{-3cm} \Big( \exp \big(\frac{1}{2}b\tau \big)\sqrt{2(B_{1}+B_{3})}+ \big(\exp \big(\frac{1}{2}b\tau \big)-1\big)\frac{2}{b\sqrt{\kappa(\alpha)}}(B_{2}(h)+B_{4}\|h\|_{\infty}^{\frac{1}{2}})\Big)
  \frac{1}{\sqrt{\kappa(\alpha)}}\|h\|_{\infty}^{\frac{3}{2}}.
 \end{eqnarray*}
 Using the Cauchy Schwarz inequality and integrating over $\tau\in[0,T]$ delivers as in the proof of Theorem \ref{gateaux-stetig} 
 \begin{eqnarray}\label{dL2}
  \|d\|_{L^{2}(0,T;U)}^{2}
  %&&=\int\limits_{0}^{T}\|d(\tau,\cdot)\|_{H_{0}^{1}(\Omega,\RR^3)}^{2}d\tau \nonumber \\
  &\leq & \\
  &&\hspace*{-2.5cm} 2\|h\|_{\infty}^{3}\Big(\frac{2(B_{1}+B_{3})}{\kappa(\alpha)}\int\limits_{0}^{T}\exp(b\tau)\,d\tau +
  \frac{4(B_{2}(h)+B_{4}\|h\|_{\infty}^{\frac{1}{2}})^{2}}{b^{2}\kappa(\alpha)^{2}}\int\limits_{0}^{T}(\exp(\frac{1}{2}b\tau)-1)^{2}\,d\tau\Big). \nonumber
 \end{eqnarray}
 In this way we obtain the final and desired estimate
 \begin{equation}\label{final-estimate}    \|d\|_{L^{2}(0,T;U)} \leq \bar{L}_{2} (h)\, \|h\|_{\infty}^{\frac{3}{2}} \end{equation}
 with
 \begin{equation*}
  \bar{L}_{2} (h):= 
  \sqrt{2}\Big(\frac{2(B_{1}+B_{3})}{\kappa(\alpha)}\exp(bT)T + \frac{4(B_{2}(h)+B_{4}\|h\|_{\infty}^{\frac{1}{2}})^{2}}{b^{2}\kappa(\alpha)^{2}}\big(\exp(\frac{1}{2}bT)-1\big)^{2}T\Big)^{\frac{1}{2}}>0.
  \end{equation*}
  We note, that because of (\ref{B-h-zero}), we have the convergence
  \begin{equation}\label{L2-convergence}
  \lim_{h\to 0} \bar{L}_2 (h) = 
  \sqrt{2}\Big(\frac{2(B_{1}+B_{3})}{\kappa(\alpha)}\exp(bT)T + \frac{4 \bar{B}_{2}^2}{b^{2}\kappa(\alpha)^{2}}\big(\exp(\frac{1}{2}bT)-1\big)^{2}T\Big)^{\frac{1}{2}}>0
  \end{equation}
  with $\bar{B}_2$ from (\ref{B-h-zero}). Applying Poincar\'{e}'s inequality and taking the convergence (\ref{L2-convergence}) as well as (\ref{final-estimate}) into account, 
  we find for $\|h\|_\infty$ sufficiently small an appropriate constant $L_2>0$ such that 
  \[\|d\|_{L^{2}(0,T;V)} \leq L_{2}\|h\|_{\infty}^{\frac{3}{2}}\] 
  is valid for all $h\to 0$. This is the assertion of Theorem \ref{thm-glm-konv} and the proof is complete.\\
\end{proof}
\end{appendix}

%%%%%%%%%%%%%%%%%%%%%%%%%%%%%%%%%%%%%%%%%%%%

\section*{Acknowledgments}

The work was partly funded by German Science Foundation (Deutsche Forschungsgemeinschaft, DFG) under Schu 1978/4-2.\\

\bibliographystyle{siam}
\bibliography{forces-composites-bib,references-hyperelastic,references}

%\end{otherlanguage} 

\end{document}